\documentclass[11pt]{amsart}
\usepackage{amsfonts,amssymb,amscd,amsmath,enumerate,verbatim,calc,graphicx}
\usepackage[all]{xy}
\usepackage[margin=2.4cm]{geometry}
\usepackage[usenames,dvipsnames]{color}
\usepackage{booktabs, colortbl, enumitem, color}

\def\NZQ{\mathbb}               
  \def\ZZ{{\NZQ Z}}    \def\HH{{\NZQ H}}
\def\opn#1#2{\def#1{\operatorname{#2}}} 
\opn\GL{GL} \opn\auto{Aut} \opn\inn{Inn} \opn\fix{Fix} \opn\ord{ord} \opn\id{id} \opn\lcm{lcm} \opn\TN{TN} 
\opn\Ker{Ker} \opn\Im{Im} \opn\GA{GAQ} \opn\rank{rank} \opn\Dis{Dis}
   \def\Qc{{\mathcal Q}}
\newcommand{\Sf}{\mathfrak{S}}
\newcommand{\Af}{\mathfrak{A}}
\newcommand{\Dic}{\mathrm{Dic}}
\newcommand{\piconj}{a}
\newcommand{\ginn}{\inn_{\mathrm{gr}}}
\newcommand{\qinn}{\inn_{\mathrm{qu}}}
\newcommand{\gauto}{\auto_{\mathrm{gr}}}
\newcommand{\qauto}{\auto_{\mathrm{qu}}}

\newtheorem{thm}{Theorem}[section]
\newtheorem{lem}[thm]{Lemma}
\newtheorem{prop}[thm]{Proposition}
\newtheorem{cor}[thm]{Corollary}

\newtheorem{prob}[thm]{Problem}
\theoremstyle{definition}

\newtheorem{ex}[thm]{Example}

\newtheorem*{acknowledgement}{Acknowledgments}
\theoremstyle{remark}
\newtheorem{rem}[thm]{Remark}
\begin{document}

\title{Generalized Alexander quandles of finite groups and their characterizations}

\author{Akihiro Higashitani}
\author{Hirotake Kurihara}
\address{Akihiro Higashitani, Department of Pure and Applied Mathematics, Graduate School of Information Science and Technology, Osaka University, 
Yamadaoka 1-5, Suita, Osaka 565-0871, Japan}
\email{higashitani@ist.osaka-u.ac.jp}
\address{Hirotake Kurihara, Department of Applied Science, Yamaguchi University, 
Tokiwadai 2-16-1, Ube 755-8611, Japan}
\email{kurihara-hiro@yamaguchi-u.ac.jp}

\begin{abstract}
The goal of this paper is to characterization generalized Alexander quandles of finite groups in the language of the underlying groups. 
Firstly, we prove that if finite groups $G$ are simple, then the quandle isomorphic classes of generalized Alexander quandles of $G$ 
one-to-one correspond to the conjugacy classes of the automorphism groups of $G$. 
This correspondence can be also claimed for the case of symmetric groups. 
Secondly, we give a characterization of generalized Alexander quandles of finite groups $G$ under some assumptions in terms of $G$. 
As corollaries of this characterization, we obtain several characterizations in some particular groups, e.g., abelian groups and dihedral groups. 
Finally, we perform a characterization of generalized Alexander quandles arising from groups with their order up to $15$. 
\end{abstract}

\maketitle

\tableofcontents

\section{Introduction}

\subsection{Backgrounds and motivations}

Let us recall what a quandle is. 
A set $Q$ equipped with a binary operation $*$ is said to be a \textit{quandle} if the following three conditions are satisfied: 
\begin{itemize}
\item[(Q1)] $x*x=x$ for any $x \in Q$; 
\item[(Q2)] for each $x,y \in Q$, there is a unique $z \in Q$ such that $z*y=x$; 
\item[(Q3)] for any $x,y,z \in Q$, we have $(x*y)*z=(x*z)*(y*z)$. 
\end{itemize}
A notion of quandles was originally introduced in \cite{Joyce} in the context of knot theory. 
In fact, these axioms correspond to the Reidemeister moves. 

In this paper, we employ the following definition of quandles using ``point symmetry'', which is a map globally defined at each point. 
We say that a set $Q$ equipped with a point symmetry $s_x : Q \rightarrow Q$ for each $x \in Q$ is a quandle if the following three conditions are satisfied: 
\begin{itemize}
\item[(Q1')] $s_x(x)=x$ for each $x \in Q$; 
\item[(Q2')] $s_x$ is a bijection on $Q$ for each $x \in Q$; 
\item[(Q3')] $s_x \circ s_y=s_{s_{x(y)}} \circ x_x$ holds for each $x,y \in Q$. 
\end{itemize}
We can directly verify that those definitions are equivalent by letting $s_x( \cdot ) : Q \rightarrow Q$ be $(\cdot) * x$. 
To be precise, we should denote a quandle $Q$ by $(Q,*)$ or $(Q,s)$, i.e., a pair of a set $Q$ and a binary operation $*$ or a point symmetry $s_x( \cdot ) (x \in Q)$,  
but we often write it for $Q$ if there is no confusion, while we also write, e.g., $(Q,s)$ or $(Q,s')$ if we want to distinguish the equipped operations. 

The above equivalent definition derives from the theory of symmetric spaces since quandles can be regarded as symmetric spaces ``without differential structure''. 
In particular, due to the similarity of the theory of homogeneous spaces, homogeneous quandles have been intensively studies 
(see, e.g., \cite{Bonatto, IT,KTW,Sin,Tam,Ven,Wada}). 
Here, a quandle $Q$ is said to be \textit{homogeneous} if the automorphism group of $Q$ acts transitively on $Q$. 
Moreover, it is proved that any homogeneous quandles can be obtained 
from a triple of a group $G$, its subgroup $K$ and its automorphism $\psi$ satisfying a certain condition, called a \textit{quandle triplet}. 
(See \cite[Section 3]{IT} and \cite[Section 7]{Joyce}.) 
In particular, in the case where $K=\{e\}$, we can always construct a certain homogeneous quandle, which we will denote $Q(G,\psi)$. 
This homogeneous quandle is called a \textit{generalized Alexander quandle} (also known as a \textit{principal quandle}) of $G$ and $\psi$. 
See Subsection~\ref{subsec:genAlex} for more details.
Note that Alexander quandles correspond to the case where $G$ is abelian. 

The main object of this paper is generalized Alexander quandles $Q(G,\psi)$ for finite groups $G$. 
The symbols that appear in the following sections will be explained in Section~\ref{sec:notation}.

Given a finite group $G$, let $\Qc(G)$ be the set of quandle isomorphic classes of $Q(G,\psi)$'s, i.e., $$\Qc(G):=\{Q(G,\psi) : \psi \in \gauto(G)\}/\cong_\mathrm{qu}.$$
The following problem naturally arises. 
\begin{prob}[{\cite[Problem 3.5]{HK}}]\label{toi}
Determine $\Qc(G)$ for a given group $G$. 
\end{prob}

Nelson completely solves this problem in the case of finite abelian groups in \cite[Theorem 2.1]{Nelson}. 
The goal of this paper is to give a more general statement for this problem. 

\bigskip

\subsection{Key notion}

Let $G$ be a finite group with its identity $e$, let $\psi$ be an automorphism of $G$ and let $Q=Q(G,\psi)$ be a generalized Alexander quandle. 
In this paper, we introduce an invariant of generalized Alexander quandle: 
\begin{equation}\label{P}
\begin{split}
P=P(Q) &:= \qinn(Q)e \\
&=\{ x\in G : \exists a_1,\ldots , \exists a_r \in G \text{ such that }x=s_{a_1} \circ \cdots \circ s_{a_r} (e)\}. 
\end{split}
\end{equation}
Namely, $P$ is the orbit of $e$ by the action of $\qinn(Q(G,\psi))$. 
Note that the same notion has already appeared in \cite{Bonatto}. See Remark~\ref{rem:Bonatto}. 

This $P$ plays a crucial role for the solution of Problem~\ref{toi}. 
In fact, $P$ satisfies several important properties on quandle isomorphism and will play key roles for the proofs of our main results.

\bigskip

\subsection{Main results}

The first main result of this paper is the following.
\begin{thm}\label{simple}
Let $G$ be a finite simple group. Then $\Qc(G)$ one-to-one corresponds to the set of conjugacy classes of $\gauto(G)$.  
\end{thm} 
Although it is well known that the symmetric group on $\{1,\ldots,n\}$ is not simple for $n \geq 3$, we immediate obtain the following as a corollary of this theorem. 
\begin{cor}\label{cor:S_n}
Let $\Sf_n$ be the symmetric group on $\{1,\ldots,n\}$. Then $\Qc(\Sf_n)$ one-to-one corresponds to the set of conjugacy classes of $\gauto(\Sf_n)$. 
\end{cor}
See Subsection~\ref{sec:proof_simple} for the proofs of Theorem~\ref{simple} and Corollary~\ref{cor:S_n}. 
Note that this corollary implies \cite[Theorem 5.14]{HK}, which claims the same result in the case $n \in \{3,4,\ldots,30\} \setminus \{15\}$, 
and gives an affirmative answer for \cite[Question 5.15]{HK}.

\bigskip

Next, we consider generalized Alexander quandles satisfying \ref{G-normal} and \ref{iikanji} 
(definitions of \ref{G-normal} and \ref{iikanji} will appear in Section~\ref{sec:PandP2}). 
We can give an equivalent condition of the quandle isomorphic property
among generalized Alexander quandles satisfying \ref{G-normal} and \ref{iikanji}.
This equivalent condition is the second main result of this paper 
and the detail is the following.
\begin{thm}\label{monosugoi}
Let $G$ and $G'$ be finite groups, let $\psi \in \gauto(G)$ and $\psi' \in \gauto(G')$, and let $Q=Q(G,\psi)$ and $Q'=Q(G',\psi')$, respectively. 
Assume that $Q$ and $Q'$ satisfy \ref{G-normal} and \ref{iikanji}. 
Then $Q \cong_\mathrm{qu} Q'$ holds if and only if the following conditions are satisfied: 
\begin{itemize}
\item[{\em (A)}] $|G|=|G'|$; 
\item[{\em (B)}] $|\fix(\psi,G)|=|\fix(\psi',G')|$; 
\item[{\em (C)}] there is a group isomorphism $h$ from $P=P(Q)$ to $P'=P(Q')$ (i.e., $P \cong_\mathrm{gr} P'$) satisfying the following conditions:  
\begin{itemize}
\item[{\em (C-1)}] $h \circ \psi|_P = \psi'|_{P'} \circ h$ and 
\item[{\em (C-2)}] for any $a \in G$ there is $a' \in G'$ such that $h(s_a(e))=s_{a'}'(e')$ holds. 
\end{itemize}
\end{itemize}
\end{thm}
This theorem induces many interesting results. 
In fact, there are many generalized Alexander quandles satisfying the conditions \ref{G-normal} and \ref{iikanji}.
The following three corollaries follow from Theorem~\ref{monosugoi}. 

\begin{cor}[{\cite[Theorem 2.1]{Nelson}}]\label{cor1}
Let $G$ and $G'$ be finite abelian groups and let $\psi \in \gauto(G)$ and $\psi' \in \gauto(G')$, respectively. 
Let $\rho=\id_G - \psi$ and let $\rho'=\id_{G'}-\psi'$. Then $Q(G,\psi) \cong_\mathrm{qu} Q(G',\psi')$ holds if and only if 
$|G|=|G'|$ holds and there is a group isomorphism $h : \Im \rho \rightarrow \Im \rho'$ satisfying $h \circ \psi|_{\Im \rho} = \psi'|_{\Im \rho'} \circ h$. 
\end{cor}
Corollary~\ref{cor1} gives a group-theoretic interpretation of \cite[Theorem 2.1]{Nelson}. 
See Section~\ref{sec:relation} for more details. 

\begin{cor}\label{cor2}
Let $D_n$ be a dihedral group of order $2n$, let $a,a' \in C_n^\times$ and let $b,b' \in C_n$.
Put $d=\gcd(n,1-a,b)$ and $d'=\gcd(n,1-a',b')$.
Then $Q(D_n,\varphi_{a,b}) \cong_\mathrm{qu} Q(D_n,\varphi_{a',b'})$ holds if and only if the following three conditions are satisfied: 
\begin{itemize}
\item $|\fix(\varphi_{a,b},D_n)|= |\fix(\varphi_{a',b'},D_n)|$; 
\item $d=d'$; 
\item $a \equiv a' \pmod{\frac{n}{d}}$. 
\end{itemize}
\end{cor}
See Section~\ref{sec:D_n} for more details. 

\begin{cor}\label{cor3}
Let $n$ be a positive integer and fix $a \in C_{2n}^\times$. Let $a=2k+1$ for some nonnegative integer $k$, let $g=\gcd(k,n)$ and let $\tilde{a}=a-\lfloor \frac{a}{n} \rfloor n$. 
Then we have a quandle isomorphism $Q(C_{2n}, a) \cong_\mathrm{qu} Q(D_n,\varphi_{\tilde{a},g})$. 
\end{cor}
This corollary means that any \textit{linear} Alexander quandle of cardinality $2n$ can be realized 
as a generalized Alexander quandle $Q(D_n,\varphi_{c,d})$ for some certain $c \in C_n^\times$ and $d \in C_n$. 
See Section~\ref{sec:relation} for more details.

\bigskip

The third main result of this paper is a complete classification of generalized Alexander quandles up to quandle isomorphism for finite groups with small orders. 
Namely, we consider the problem of the determination of $\Qc_{\GA}(n)$ for each $n$'s, where 
\begin{align*}
\Qc_{\GA}(n):=&\{Q(G,\psi) : \text{$G$ is a group of order $n$}, \; \psi\in \gauto(G)\}/\cong_\mathrm{qu} \\
=&\bigcup_{\text{$G$ is a group of order $n$}}\Qc(G).
\end{align*}
\begin{thm}
    The complete list of $|\Qc_{\GA}(n)|$ for $n \leq 15$ is given in Table~\ref{tab:GA}. 
\end{thm}
\begin{table}[tbh]
\begin{center}
\begin{tabular}{rccccccccccccccc}
\toprule
$n$ & 1 & 2 & 3 & 4 & 5 & 6 & 7 & 8 & 9 & 10 & 11 & 12 & 13 & 14 & 15\\
$|\Qc_{\GA}(n)|$ & 1 & 1 & 2 & 3 & 4 & 3 & 6 & 9 & 11 & 5 & 10& 11 & 12& 7& 8\\
\bottomrule
\end{tabular}

\medskip

\caption{List of $|\Qc_{\GA}(n)|$}\label{tab:GA}
\end{center}
\end{table}

The results in the cases $n=2,3,5,7,11,13$ come from Subsection~\ref{sec:p}, the results in the cases $n=6,10,14$ come from Subsection~\ref{sec:2p} 
and the results in the cases $n=4,9$ come from Subsection~\ref{sec:p^2}. 
The cases $n=8$, $n=12$ and $n=15$ are discussed in Subsections~\ref{sec:8}, \ref{sec:12} and \ref{sec:15}, respectively. 
Finally, we mention in Subsection~\ref{sec:16} that Theorem~\ref{monosugoi} is not available in the case  $n=16$. 

\bigskip

\subsection{Structures of this paper}

An organization of this paper is as follows. In Section~\ref{sec:notation}, we collect the notions and the notation 
and recall some previous results, which will us throughout the paper.
In Section~\ref{sec:proofs}, we give our proofs of Theorem~\ref{simple}, Corollary~\ref{cor:S_n} and Theorem~\ref{monosugoi}. 
As applications of Theorem~\ref{monosugoi}, we give the proof of Corollary~\ref{cor2} in Section~\ref{sec:D_n} 
and the proofs of Corollaries~\ref{cor1} and \ref{cor3} in Section~\ref{sec:relation}. 
In Section~\ref{sec:group}, we perform a complete classification of generalized Alexander quandles of groups with their orders at most $15$. 

\bigskip

\begin{acknowledgement}
The first named author is partially supported by JSPS Grant-in-Aid for Scientists Research (C) $\sharp$20K03513 
and the second named author is partially supported by JSPS Grant-in-Aid for Scientists Research (C) $\sharp$20K03623. 
\end{acknowledgement}

\bigskip


\section{Terminologies and notation}\label{sec:notation}

\subsection{Terminologies for quandles}
Throughout this paper, we denote by $(Q,s)$ (or $Q$, for short) 
a quandle $Q$ equipped with the maps $\{s_x : x \in Q\}$ satisfying (Q1')--(Q3') in Introduction.

Let $(Q,s)$ and $(Q',s')$ be quandles. 
A map $f:Q \rightarrow Q'$ is called a {\em quandle homomorphism} if $f$ satisfies 
$$f \circ s_x = s_{f(x)}' \circ f \;\; \text{ for any } x \in Q.$$
Moreover, $f$ is said to be a {\em quandle isomorphism} when $f$ is bijective. 
If there is a quandle isomorphism between $Q$ and $Q'$, then we say that $Q$ and $Q'$ are {\em isomorphic as quandles}, 
denoted by $Q \cong_\mathrm{qu} Q'$.

We collect the notation on quandles used in the paper: 
\begin{itemize}
\item Let $\qauto(Q,s)$ (or $\qauto(Q)$, for short) be the set of all quandle automorphisms from $(Q,s)$ to $(Q,s)$ itself, 
which is called the {\em automorphism group} of $Q$. 
Note that $s_x \in \qauto(Q)$ for any $x \in Q$. 
\item Let $\qinn(Q,s)$ (or $\qinn(Q)$, for short) be the subgroup of $\qauto(Q)$ generated by $\{s_x : x \in Q\}$, which is called the {\em inner automorphism group} of $Q$. 
It is easy to see that $\qinn(Q)$ is a normal subgroup of $\qauto(Q)$. 
\item For a quandle $Q$, we call $Q$ \textit{homogeneous} if for any $x,y \in Q$, there exists a quandle automorphism $f$ on $Q$ such that $f(x)=y$. 
This means that $\qauto(Q)$ acts transitively on $Q$. 
\item For a quandle $Q$, we call $Q$ \textit{connected} if for any $x,y \in Q$, there exists a quandle inner automorphism $f$ on $Q$ such that $f(x)=y$. 
This means that $\qinn(Q)$ acts transitively on $Q$.
\item For a quandle $(Q,s)$, let $\ord(s_x)=\min\{n \in \ZZ_{>0} : \underbrace{s_x \circ \cdots \circ s_x}_n=\id\}$. 
For a homogeneous quandle $Q$, we see that $\ord(s_x)$ is constant for any $x \in Q$. 
In fact, for any $x,y \in Q$, since there is $f \in \qauto(Q)$ such that $y=f(x)$, we see that 
if $s_x^m=\id$, then $f \circ s_x^m=s_{f(x)}^m \circ f$, so we obtain $s_{f(x)}^m=\id$, as required. 
Thus, we use the notation $\ord(Q)$ instead of $\ord(s_x)$ for $x \in Q$ in the case $Q$ is homogeneous. 
\end{itemize}

\subsection{Terminologies for groups}\label{sec:Term_group}

Let $G,G'$ be groups with its units $e,e'$, respectively. 
We denote by $G \cong_\mathrm{gr} G'$ if $G$ and $G'$ are isomorphic as groups.

We collect the notation on groups used in the paper: 
\begin{itemize}
\item We denote $H<G$ (resp. $H \triangleleft G$) if $H$ is a subgroup (resp. normal subgroup) of $G$. 
\item Let $\gauto(G)$ denote the automorphism group of $G$. 
\item Let $\ginn(G)$ denote the inner automorphism group of $G$, i.e., the set of the group automorphisms 
defined by $x \mapsto axa^{-1}$ for each $a \in G$, denoted by $(\cdot)^a$. Note that $\ginn(G)$ is a normal subgroup of $\gauto(G)$. 
We call $\psi$ an \textit{outer automorphism} if $\psi \in \gauto(G) \setminus \ginn(G)$. 
\item Let $Z(G)=\{z \in G : xz=zx \text{ for any }x \in G\}$ denote the center of $G$. 
It is well known that $\ginn(G) \cong_\mathrm{gr} G/Z(G)$. 
\item Let $\ord_G(x)$ denote the order of $x \in G$. 
\item Given $\psi \in \gauto(G)$, let $$\fix(\psi,G)=\{x \in G : \psi(x)=x\}.$$ 
\item Given a subset $K$ of $G$, let $G/K=\{[x]:x \in G\}$ denote the set of left cosets with respect to $K$. 
\item Let $C_n$ be a cyclic group of order $n$. We set $C_n=\{0,1,\ldots,n-1\}$ as a set and use $+$ as an operation. 
We perform the addition in $C_n$ modulo $n$. 
\item Let $C_n^\times$ denote the multiplicative group of $C_n$, i.e., $C_n^\times=\{a \in C_n : \gcd(a,n)=1\}$. 
It is well known that $\gauto(C_n) \cong_\mathrm{gr} C_n^\times$. 
\item Let $D_n$ be a dihedral group of order $2n$. We will give our notation for dihedral groups in the beginning of Section~\ref{sec:D_n}. 
\item Let $\Sf_n$ (resp. $\Af_n$) be the symmetric group (resp. the alternating group) on $\{1,2,\ldots,n\}$. It is well known that $\Af_n$ is a normal subgroup of $\Sf_n$. 
\end{itemize}
For the terminologies or fundamental facts on group theory, consult, e.g., \cite{Roman, Rotman}.

\subsection{Quandle triplets and generalized Alexander quandles}\label{subsec:genAlex}

Let $G$ be a group and let $K$ be a subgroup of $G$. Take $\psi \in \gauto(G)$. 
We say that a triple $(G,K,\psi)$ is a \textit{quandle triplet} if $K \subset \fix(\psi,G)$. 
Given a quandle triplet $(G,K,\psi)$, we can construct a quandle $Q(G,K,\psi)$ as follows: 
$Q(G,K,\psi)=G/K$ as a set and we define $s_{[x]}$ for each $[x] \in G/K$ by setting 
$$s_{[x]}([y])=[x\psi^{-1}(x^{-1}y)] \;\;\text{for }[y] \in G/K.$$
It is proved in \cite{Joyce} that the set $G/K$ equipped with this point symmetry becomes a homogeneous quandle. 

For the relationship between homogeneous quandles and quandle triplets, the following is known: 
\begin{prop}[{see \cite[Section 3]{IT}}]
For any homogeneous quandle $Q$, there is a quandle triplet $(G,K,\psi)$ such that $Q \cong_\mathrm{qu} Q(G,K,\psi)$. 
\end{prop}

We notice that for a given a group automorphism $\psi \in \gauto(G)$, a triplet $(G,\{e\},\psi)$ trivially becomes a quandle triplet. 
Let $Q(G,\psi)=Q(G,\{e\},\psi)$, i.e., $Q(G,\psi)$ is the homogeneous quandle $(G,s)$, where $s$ is defined by 
$$s_x(y)=x\psi(x^{-1}y)\;\;\text{ for any }x,y \in G.$$ 
This quandle $Q(G,\psi)$ is known as a {\em generalized Alexander quandle}. 
Note that in the case $G$ is abelian, $Q(G,\psi)$ is called an {\em Alexander quandle} (see \cite{Nelson}).

Towards a solution of Problem \ref{toi}, 
we collect some invariants on generalized Alexander quandles developed in \cite{HK}. 
Let $\psi,\psi' \in \gauto(G)$ and let $Q=Q(G,\psi)$ and let $Q'=Q(G,\psi')$. 
The following assertions are claimed in \cite[Theorem 4.5]{HK}: 
\begin{enumerate}[label=$($\alph*$)$]
    \item If $\psi'=\psi^\tau$ for some $\tau \in \gauto(G)$, then $Q \cong_\mathrm{qu} Q'$. 
    \item If $Q \cong_\mathrm{qu} Q'$, then 
    \begin{itemize}
        \item $\qinn(Q) \cong_\mathrm{gr} \qinn(Q')$;
        \item $\ord_{\gauto(G)}\psi=\ord_{\gauto(G)}\psi'$; 
        \item $|\fix(\psi,G)|=|\fix(\psi',G)|$. 
        \end{itemize}        
\end{enumerate}

\begin{rem}
    \label{rem:GneqG'}
    In the above argument, we may obtain the similar results when the source group and the target group are different, i.e.,
    for $Q=Q(G,\psi)$ and $Q'=Q(G',\psi')$, if $Q \cong_\mathrm{qu} Q'$ then 
    \begin{itemize}
        \item $\qinn(Q) \cong_\mathrm{gr} \qinn(Q')$; 
        \item $\ord_{\gauto(G)}\psi=\ord_{\gauto(G')}\psi'$; 
        \item $|\fix(\psi,G)|=|\fix(\psi',G')|$.  
        \end{itemize}
\end{rem}

\bigskip


\section{Proofs of the main results}\label{sec:proofs}

In this section, we develop a new invariant $P$ (see \eqref{P}) on generalized Alexander quandles. 

Let $Q=Q(G,\psi)$ be a generalized Alexander quandle, where $G$ is a finite group and $\psi \in \gauto(G)$. 
Given $a \in G$, let $$L_a:G \rightarrow G\text{ defined by }L_a(x)=ax\text{ for }x \in Q.$$ 
Then we see that $L_a \in \qauto(Q)$. In fact, $L_a$ is a bijection on $Q$ since $Q=G$ as a set, and 
$$ L_a \circ s_x(y)=ax\psi(x^{-1}y)=ax\psi(x^{-1}a^{-1}ay)=s_{ax}(ay)=s_{L_a(x)} \circ L_a(y), $$
i.e., $L_a$ is a quandle homomorphism. Note that $L_a^{-1}=L_{a^{-1}}$.

\subsection{Normality for $P$ and $P^2$} 
Given a generalized Alexander quandle $Q=Q(G,\psi)$, let us recall the definition of $P(Q)$ given in Introduction. 
We discuss some fundamental properties of $P$.

\begin{lem}\label{lem_Lx}
    Let $x \in P$. 
    Then there exist $a_1,\ldots,a_r \in G$ such that $L_x = s_{a_1} \circ \cdots \circ s_{a_r}$.  
    \end{lem}
    \begin{proof}
    In the definition of $P$,
    we can add the assumption $r \equiv 0 \pmod{\ord_{\gauto(G)}\psi}$
    since $s_{a_1} \circ \cdots \circ s_{a_t}(e)=s_{a_1} \circ \cdots \circ s_{a_t} \circ s_e^m(e)$ holds for any $m \in \ZZ$. 
    Hence, for $x \in P$, there exist $a_1,\ldots ,a_r \in G$ such that $x=s_{a_1} \circ \cdots \circ s_{a_r}(e)$, where we let $r \equiv 0 \pmod{\ord_{\gauto(G)}\psi}$. 
    Thus, for any $y \in G$, we see the following: 
    \begin{align*}
    L_x(y)=xy =s_{a_1} \circ \cdots \circ s_{a_r}(e)\cdot \psi^r(y) =s_{a_1} \circ \cdots \circ s_{a_r}(y). 
    \end{align*}
    \end{proof}

In what follows, for a map $f : A \rightarrow B$ and a subset $A' \subset A$, let $f|_{A'}$ denote the restriction of $f$ into $A'$. 

The following proposition claims a property on $P$ as quandles. 
\begin{prop}\label{prop_P}
{\em (i)} $P$ becomes a subquandle of $Q(G,\psi)$. \\
{\em (ii)} $\psi|_P \in \qauto(P)$. 
\end{prop}
\begin{proof}
(i) Take $x,y \in P$ arbitrarily. Then there exist $a_1,\ldots,a_r \in G$ such that $y=s_{a_1} \circ \cdots \circ s_{a_r}(e)$. 
Thus $s_x(y)=s_x \circ s_{a_1} \circ \cdots \circ s_{a_r}(e)$, i.e., $s_x(y) \in P$, as required. 

\noindent
(ii) It is enough to show that $\psi(P) \subset P$, but this can be shown straightforwardly from $\psi=s_e$ as maps. 
\end{proof}

Since $e=s_e(e)\in P$ and by Proposition~\ref{prop_P},
$P$ becomes a subquandle of $Q(G,\psi)$ containing $e$.
Thus we can also define the following: 
\begin{equation}\label{P2}
\begin{split}
P^2=P(P(Q)) &:= \qinn(P(Q))e \\
&=\{ x\in P : \exists p_1,\ldots, \exists p_t \in P \text{ such that }x=s_{p_1} \circ \cdots \circ s_{p_t} (e)\}. 
\end{split}
\end{equation}

The following proposition claims a property on $P$ and $P^2$ as groups. 
\begin{prop}\label{normal_subgroup}
{\em (i)} $P$ is a normal subgroup of $G$. \\
{\em (ii)} $P^2$ is a normal subgroup of $P$. 
\end{prop}
\begin{proof}
The statement (ii) immediately follows from (i) together with Proposition~\ref{prop_P}. 
In fact, since $P$ is a group by (i) and $P$ also becomes a quandle by Proposition~\ref{prop_P}, 
the statement (ii) is nothing but the statement for the generalized Alexander quandle $Q(P,\psi|_P)$. 
Thus, it suffices to show (i). 

\noindent
\underline{$e \in P$}: We have already seen before the definition of \eqref{P2}.

\noindent
\underline{$x,y\in P \  \Rightarrow \  xy \in P$}:
Take $x,y\in P$. Then there exist $b_1,b_2,\ldots ,b_t \in G$ such that $y=s_{b_1}\circ \cdots \circ s_{b_t} (e)$.
Moreover, by Lemma~\ref{lem_Lx}, there exist $a_1,a_2,\ldots ,a_r  \in G$ such that
$L_x=s_{a_1}\circ \cdots \circ s_{a_r}$. Then
\[
    xy = L_x(y)
    = s_{a_1} \circ \cdots \circ s_{a_r} \circ s_{b_1}\circ \cdots \circ s_{b_t} (e).
\]
Hence we have $xy\in P$.

\noindent
\underline{$x\in P \  \Rightarrow \  x^{-1} \in P$}: Let $m=\ord_{\gauto(G)}(\psi)$. Take $x\in P$. 
Then there exist $a_1,a_2,\ldots ,a_r  \in G$ such that $L_x=s_{a_1}\circ \cdots \circ s_{a_r}$.
Since $s_a^m=\id$ for any $a \in G$, we have $L_x^{-1}=s_{a_r}^{m-1}\circ \cdots \circ s_{a_1}^{m-1}$.
By using this, we have 
$$x^{-1} = x^{-1}\cdot e =L_{x^{-1}}(e)=L_x^{-1}(e)=
s_{a_r}^{m-1}\circ \cdots \circ s_{a_1}^{m-1}(e).$$
Hence we have $x^{-1}\in P$. 

\noindent
\underline{Normal subgroup}: Now we will show $a x a^{-1}\in P$ for any $x\in P$ and $a\in G$.
By Lemma~\ref{lem_Lx},
there exist $a_1,\ldots ,a_r  \in G$ such that $L_x=s_{a_1}\circ \cdots \circ s_{a_r}$. 
Since
\[
    a x a^{-1}
    =L_a \circ L_x(a^{-1})
    =L_a \circ s_{a_1}\circ \cdots \circ s_{a_r}(a^{-1})
    =s_{a\cdot a_1}\circ \cdots \circ s_{a\cdot a_r}(a\cdot a^{-1})
    =s_{a\cdot a_1}\circ \cdots \circ s_{a\cdot a_r}(e),
\]
we have $a x a^{-1} \in P$.
\end{proof}
\begin{rem}\label{chuui_P}
Let $Q=Q(G,\psi)$ be a generalized Alexander quandle. \\
{\rm (i)} By direct computation, $P(Q)=\{e\}$ holds if and only if $\fix(\psi,G)=G$, i.e., $\psi=\id$.\\
{\rm (ii)} By definition of \eqref{P}, $P(Q)=G$ holds if and only if $Q$ is connected.
\end{rem}

\begin{rem}\label{rem:Bonatto}
In \cite{Bonatto}, the notion of the \emph{displacement group} $\Dis(Q):=\langle \{s_x\circ s_y^{-1} : x,y\in Q\}\rangle$ for quandles $Q$ 
and its importance are mentioned. 
In the context of generalized Alexander quandles $Q=Q(G,\psi)$, where $G$ is a finite group and $\psi \in \gauto(G)$, the following is proved: 
$$\Dis(Q)\cong_\mathrm{gr}[G,\psi]=\langle \{x\psi(x)^{-1} : x \in G\}\rangle.$$
By Proposition~\ref{normal_subgroup}, we can prove that $[G,\psi]\cong_\mathrm{gr} P(Q)$.
In fact, since $x\psi(x)^{-1}=s_x(e)$ for $x \in G$ and $P(Q)$ is a subgroup of $G$, it follows that $[G,\psi] \subset P(Q)$. 
Moreover, since $s_x \circ s_y(e)=s_{s_x(y)}(s_x(e))=s_{s_x(y)}(e) \cdot s_{\psi(x)}(e)$ for $x,y \in G$, we obtain the other inclusion. 
\end{rem}

\begin{cor}\label{cor:PeqG}
If $G$ is simple and $\psi \neq \id$, then $P=G$.
\end{cor}
\begin{proof}
Since $\psi \neq \id$,  we have $P \neq \{e\}$ by Remark~\ref{chuui_P} (i). 
Since $G$ is simple and $P$ is a normal subgroup of $G$ by Proposition~\ref{normal_subgroup} (i), we obtain that $P=G$. 
\end{proof}

\begin{rem}
By Remark~\ref{chuui_P} (ii) and Corollary~\ref{cor:PeqG}, 
generalized Alexander quandles obtained from finite simple groups and non-identity automorphisms are connected.
\end{rem}

\subsection{Description of $\qinn(Q(G,\psi))$} 
Next, we describe $\qinn(Q(G,\psi))$ by using $P$. 
\begin{prop}[cf. {\cite[Lemma 2.7]{Bonatto}}]\label{prop:inn(Q)}
Let $G$ be a finite group and let $\psi \in \gauto(G)$.
Let $P=P(Q(G,\psi))$ and $m=\ord_{\gauto(G)}(\psi)$. 
Then we have $$\qinn(Q(G,\psi)) \cong_\mathrm{gr} P \rtimes_\phi C_m,$$ 
which is a semidirect product of $P$ and $C_m$ with $\phi : C_m \rightarrow \gauto(G)$, $i \mapsto \psi^i$, i.e., 
$$(x,i) \cdot (y,j)=(x \psi^i(y),i+j).$$
\end{prop}

As in \cite[Proposition~5.6]{HK}, the semi-direct product in Proposition~\ref{prop:inn(Q)} may become a direct product in some cases.
The following proposition shows the relationship between the structure of the semi-direct product and the property of $\psi|_P$ when $P$ is centerless, i.e., $Z(P)=\{e\}$. 

\begin{prop}
\label{prop:semidirect}
Work with the notation as above. Assume that $P$ is centerless. 
Then the following assertions hold. \\
{\em (i)} If $\psi|_P\in \ginn(P)$, i.e., $\psi|_P$ is an inner automorphism, then we have $$\qinn(Q(G,\psi)) \cong_\mathrm{gr} P \rtimes_\phi C_m \cong_\mathrm{gr} P \times C_m.$$
{\em (ii)} If $\psi|_P\notin \ginn(P)$, i.e., $\psi|_P$ is an outer automorphism, then we have $$\qinn(Q(G,\psi)) \cong_\mathrm{gr} P \rtimes_\phi C_m \not\cong_\mathrm{gr} P \times C_m.$$
\end{prop}

\begin{proof}
(i)
By $\psi|_P \in \ginn(P)$, there exists $\piconj\in P$ such that $\psi|_P=(\cdot)^\piconj$.
Since $P$ is centerless, $\ginn(P)\cong_\mathrm{gr} P$ holds. Then we have $\ord_{\gauto(P)}(\psi|_P)=\ord_P(\piconj)$.
On the other hand, $\ord_{\gauto(P)}(\psi|_P)$ is a divisor of $m=\ord_{\gauto(G)}(\psi)$.
Hence $\ord_P(\piconj)$ is a divisor of $m$. 
We define the map $\Psi : P \rtimes_\phi C_m \rightarrow P \times C_m$ by 
$$\Psi((x,i)) =(x \piconj^i, i).$$ 
Then it is straightforward to check that this gives a group isomorphism: 
\begin{itemize}
\item Since $\piconj \in P$, we see that $x \piconj^i \in P$ for any $x \in P$ and $i \in C_m$. 
Also since $\ord_P(a)$ is a divisor of $m$, we see that $x \piconj^i=x \piconj^j$ whenever $i \equiv j \pmod{m}$. Hence, the well-definedness follows. 
\item One has $\Psi((x,i)\cdot(y,j))=\Psi((x\piconj^i y \piconj^{-i},i+j))=(x\piconj^iy\piconj^j,i+j)=\Psi((x,i))\Psi((y,j)).$ 
\item One has $\Psi((x,i))=\Psi((y,j)) \;\Longleftrightarrow\;(x\piconj^i,i)=(y\piconj^j,j)  \;\Longleftrightarrow\; x=y \text{ and } i \equiv j \pmod{m}$. 
\item For any $(x,i) \in P \times C_m$, since $x \piconj^{-i} \in P$, we have $\Psi((x\piconj^{-i},i))=(x,i)$, as required. 
\end{itemize}

\medskip

(ii)
Suppose that there exists a group isomorphism $\Omega \colon P \rtimes_\phi C_m \to P \times C_m$. 
Let $X=\{(x, 0) : x \in P\} \subset P \rtimes_\phi C_m$. Then $X$ is a subgroup of $P \rtimes_\phi C_m$. 
Let $X'=\Omega(X)$. Then $X'$ becomes a subgroup of $P \times C_m$. 
For each $(x, 0) \in X$, we write $(a_x,j_x)=\Omega(x,0)$.

Let $Y'=\{(e,i) : i \in C_m\} \subset P \times C_m$. Then $Y'$ is a subgroup of $P \times C_m$. 
Let $Y=\Omega^{-1}(Y')$. Then $Y$ becomes a subgroup of $P \rtimes_\phi C_m$ and $|Y|=|Y'|=m$ since $\Omega$ is a group isomorphism. 
For each $(e,i) \in Y'$, we write $(b_i,k_i)=\Omega^{-1}((e,i))$. 

Since $\Omega^{-1}$ is also a group isomorphism, we see that 
\begin{align*}
\Omega^{-1} ((e,i) (a_x,j_x))
&=\Omega^{-1} ((e,i)) \cdot \Omega^{-1} ((a_x,j_x))
=(b_i,k_i)\cdot(x,0)
=(b_i\psi^{k_i} (x) ,k_i) \text{ and }\\
\Omega^{-1} ((e,i) (a_x,j_x))&=\Omega^{-1} ((a_x,i+j_x))=\Omega^{-1} ((a_x,j_x)(e,i))
=\Omega^{-1} ((a_x,j_x))\cdot \Omega^{-1} ((e,i))\\
&=(x,0) \cdot (b_i,k_i)=(xb_i,k_i). 
\end{align*}
Hence, we obtain that $b_i\psi^{k_i} (x) = xb_i \;\Longleftrightarrow \psi^{k_i} (x) = b_i^{-1}xb_i$
for any $x \in P$, i.e., $\psi^{k_i}|_P$ is an inner automorphism on $P$. 
Note that $b_i$ is determined by $\psi$ and $k_i$ since $P$ is centerless.
We put $b_i=b_{k_i}$. 
This implies that 
$$Y=\{(b_i,k_i) : i \in C_m\} =\{(b_k,k) : k\in C_m\text{ and }\psi^{k}|_P \in \ginn(P)\}.$$
On the other hand, $\psi|_P \not\in \ginn(P)$. This implies $k \neq 1$. Thus, we have $|Y|\le m-1$. 
This is a contradiction to $|Y|=m$.

Therefore, $P \rtimes_\phi C_m \not\cong_\mathrm{gr} P \times C_m$, as desired.
\end{proof}
When $P$ is not centerless, Proposition~\ref{prop:semidirect} does not hold in general. 
In fact, let $G$ be a special linear group $\mathrm{SL}(2,3)$ of degree $2$ over $\mathbb{F}_3$ and 
let $\psi$ be the inner automorphism on $G$ with respect to $A=\begin{pmatrix} 0 & -1\\ 1 & 0 \end{pmatrix}$. 
Note that $\ord_{\auto(G)}(\psi)=2$, nevertheless $\ord_{G}(A)=4$. 
Then $P=P(Q(G,\psi))\cong_\mathrm{gr} Q_8$, where $Q_8$ is the quaternion group. (For the detail of $Q_8$, see Section~\ref{sec:8}.) 
It is well known that $Z(Q_8) \cong_\mathrm{gr} C_2$. Moreover, $\psi|_P \in \ginn(P)$.
On the other hand, we can verify that $$\qinn(Q(G,\psi)) \cong_\mathrm{gr} P \rtimes C_2 \not\cong_\mathrm{gr} P \times C_2.$$

\subsection{Proofs of Theorem~\ref{simple} and Corollary~\ref{cor:S_n}}\label{sec:proof_simple}
Now, we prove the following theorem, which will play a crucial role for the proof of Theorem~\ref{simple}. 
\begin{thm}\label{breakthrough}
Let $G,G'$ be finite groups with the identity elements $e,e'$,
respectively.
Let $\psi \in \gauto(G)$ and $\psi' \in \gauto(G')$ and 
$P=P(Q(G,\psi))$ and $P'=P(Q(G',\psi'))$. 
Assume that $Q(G,\psi) \cong_\mathrm{qu} Q(G',\psi')$ and
let $f$ be a quandle isomorphism $f:Q(G,\psi) \rightarrow Q(G',\psi')$ with $f(e)=e'$.
Then the following assertions hold: 
\begin{itemize}
\item[{\em (i)}] The restriction $f|_P$ gives a quandle isomorphism between $Q(P,\psi|_P)$ and $Q(P',\psi'|_{P'})$. 
\item[{\em (ii)}] The equality $f \circ \psi = \psi' \circ f$ holds. 
\item[{\em (iii)}] A quandle isomorphism $f|_P$ is a group homomorphism. In particular, $P \cong_\mathrm{gr} P'$ and $f|_P \in \gauto(P)$. 
\item[{\em (iv)}] For any $x \in G$, $f(xP)=f(x)P'$ holds.  
\end{itemize}
\end{thm}
Note that when $Q(G,\psi) \cong_\mathrm{qu} Q(G',\psi')$, we can always choose a quandle isomorphism $f$ between them with $f(e)=e'$ (see, e.g.,~\cite{HK}).
We also remark that $Q(P,\psi|_P)$ (resp. $Q(P',\psi'|_{P'})$) becomes a subquandle of $Q(G,\psi)$ (resp. $Q(G',\psi')$) by Proposition~\ref{prop_P} (i).

\begin{proof}[Proof of Theorem~\ref{breakthrough}] 

Let $s$ (resp. $s'$) be the point symmetry of $Q(G,\psi)$ (resp. $Q(G',\psi')$).
\noindent
(i) Since $f|_P$ is a quandle homomorphism, 
it is enough to show that $f|_P(P) = P'$.
This can be shown by 
$$
f( s_{a_1} \circ \cdots \circ s_{a_r}(e) ) =s_{f(a_1)}' \circ \cdots \circ s_{f(a_r)}'(e')
\;\text{ and }
s'_{a'_1} \circ \cdots \circ s_{a'_t}(e')=f(s_{f^{-1}(a'_1)} \circ \cdots \circ s_{f^{-1}(a'_t)}(e)) 
$$
since $f$ is a quandle isomorphism with $f(e)=e'$. 

\noindent
(ii) Note that $\psi=s_e$ holds as maps. Since $f$ is a quandle homomorphism, we have $f \circ s_e = s_{f(e)}' \circ f$, i.e., $f \circ \psi = \psi' \circ f$. 

\noindent
(iii) Since $f|_P$ is a bijection between $P$ and $P'$ by (i), it is enough to show that $f|_P$ becomes a group homomorphism. 
For any $x \in P$, there exist $a_1,\ldots,a_r \in G$ such that $x=s_{a_1} \circ \cdots \circ s_{a_r}(e)$.
As we have already seen, 
\begin{align*}
f(x)=f \circ s_{a_1} \circ \cdots \circ s_{a_r}(e) = s_{f(a_1)}' \circ \cdots \circ s_{f(a_r)}'(e') 
\end{align*}
holds. Thus, for any $x,y \in P$, we have the following: 
\begin{align*}
f(xy)&=f \circ L_x(y)=f \circ s_{a_1} \circ \cdots \circ s_{a_r} (y) \;\; \text{ (by Lemma~\ref{lem_Lx})} \\
&=s_{f(a_1)}' \circ \cdots \circ s_{f(a_r)}' (f(y)) \\
&=s_{f(a_1)}' \circ \cdots \circ s_{f(a_r)}'(e') \cdot f(y) \;\; \text{ ($\ord_{\gauto(G)}\psi=\ord_{\gauto(G')}\psi'$ by Remark~\ref{rem:GneqG'})} \\
&=f(x)f(y).  
\end{align*}

\noindent
(iv) Firstly, we show $f(xP)\subset f(x)P'$. 
For any $x \in G$ and $p=s_{a_1} \circ \cdots \circ s_{a_r}(e) \in P$, we have
\begin{align*}
    f(xp) =& f\circ L_x\circ s_{a_1} \circ \cdots \circ s_{a_r}(e) 
    =
    f\circ  s_{xa_1} \circ \cdots \circ s_{xa_r}(x)\\
    =&
    s'_{f(xa_1)} \circ \cdots \circ s'_{f(xa_r)}(f(x))
    =
    L_{f(x)}\circ s'_{f(x)^{-1}f(xa_1)} \circ \cdots \circ s'_{f(x)^{-1}f(xa_r)}(e')\\
    =&
f(x)\cdot  s'_{f(x)^{-1}f(xa_1)} \circ \cdots \circ s'_{f(x)^{-1}f(xa_r)}(e') \in f(x)P'.
\end{align*}

Next, we show $f(xP)\supset f(x)P'$. 
For any $x \in G$ and $p'=s'_{a'_1} \circ \cdots \circ s'_{a'_r}(e') \in P'$, we have
\begin{align*}
    f(x)p' =& L_{f(x)}\circ  s'_{a'_1} \circ \cdots \circ s'_{a'_r}(e')
    =s'_{f(x)a'_1} \circ \cdots \circ s'_{f(x)a'_r}(f(x))\\
    =&s'_{f(x)a'_1} \circ \cdots \circ s'_{f(x)a'_r}\circ f(x)
    =f\circ s_{f^{-1}(f(x)a'_1)} \circ \cdots \circ s_{f^{-1}(f(x)a'_r)}(x)\\
    =&f\circ L_x \circ s_{x^{-1}f^{-1}(f(x)a'_1)} \circ \cdots \circ s_{x^{-1}f^{-1}(f(x)a'_r)}(e)\in f(xP).
\end{align*}
\end{proof}

By Theorem~\ref{breakthrough}, we can prove Theorem~\ref{simple}. 
\begin{proof}[Proof of Theorem~\ref{simple}]
Let $\psi,\psi' \in \gauto(G)$ and let $Q=Q(G,\psi),Q'=Q(G,\psi')$, respectively. 
Assume that $Q$ and $Q'$ are isomorphic as quandles. 
By Remark~\ref{chuui_P}, we do not consider the case where $\psi$ or $\psi'$ is identity.
Since $G$ is simple, we see that $G=P(Q)=P(Q')$ by Corollary~\ref{cor:PeqG}. 
Thus, $Q=P(Q)$ (resp. $Q'=P(Q')$) as sets.

Let $f:Q \rightarrow Q'$ be a quandle isomorphism with $f(e)=e$. 
Then it follows from Theorem~\ref{breakthrough} (iii) that $f$ is a group automorphism of $G$. 
Moreover, by Theorem~\ref{breakthrough} (ii), we have $f \circ \psi= \psi' \circ f$, i.e., $f \circ \psi \circ f^{-1} = \psi'$. 
This says that $\psi$ and $\psi'$ are conjugate in $\gauto(G)$. 
\end{proof}

\begin{proof}[Proof of Corollary~\ref{cor:S_n}]
When $n\le 4$, it was already proved that $\Qc(\Sf_n)$ one-to-one corresponds to the set of conjugacy classes of $\gauto(\Sf_n)$ in \cite{HK}. 
Thus, we may assume $n\ge 5$. 

Let $\psi,\psi' \in \gauto(\Sf_n) \setminus \{\id\}$ and let $Q=Q(\Sf_n,\psi),Q'=Q(\Sf_n,\psi')$, respectively. 
Since $x$ and $\psi(x)$ have the same parity for any $x\in \Sf_n$, the parity of $x\psi(x^{-1})$ is even, i.e., $P=P(Q) \subset \Af_n$.
It is known that $\Af_n$ is the unique non-trivial normal subgroup of $\Sf_n$ if $n\ge 5$. 
Hence, $P=\Af_n$ holds if $\psi \neq \id$ (see Remark~\ref{chuui_P}). Similarly, we also have $P'=P(Q')=\Af_n$. 

Assume that $Q$ and $Q'$ are isomorphic as quandles. 
Since $\Af_n$ is simple for $n\ge 5$, $\psi|_{\Af_n}$ and $\psi'|_{\Af_n}$ are conjugate in $\gauto(\Af_n)$ by Theorem~\ref{simple}. 
Hence, we can take $h\in \gauto(\Af_n)$ such that $h \circ \psi|_{\Af_n}= \psi'|_{\Af_n} \circ h$. 
It is known that the restriction map $f \mapsto f|_{\Af_n}$ from $\gauto(\Sf_n)$ to $\gauto(\Af_n)$ gives an isomorphism between them (e.g.~\cite[11.4.1]{Scott}). 
Thus, there exists $f\in \gauto(\Sf_n)$ such that $f|_{\Af_n} =h$. Therefore for any $x\in \Af_n$, we have
\begin{align*}
h \circ \psi|_{\Af_n} (x)= \psi'|_{\Af_n} \circ h (x) \;\Longleftrightarrow  f \circ \psi (x)= \psi' \circ f (x) \Longleftrightarrow f \circ \psi \circ f^{-1} (x)=\psi' (x).
\end{align*}
Once we can prove that $f \circ \psi \circ f^{-1} (x)=\psi' (x)$ holds for any $x \in \Sf_n$, we obtain that $\psi$ and $\psi'$ are conjugate. 
Let $\varphi=\psi'^{-1} \circ f \circ \psi \circ f^{-1} \in \gauto(\Sf_n)$. It suffices to show that $\varphi(x)=x$ holds for any $x \in \Sf_n \setminus \Af_n$. 

Given $x,y\in \Sf_n\setminus \Af_n$, i.e., $x$ and $y$ are odd, we know that $x^{-1} y$ is even. Thus we have
\[  x^{-1} y =\varphi (x^{-1} y) \Longleftrightarrow \varphi(x) x^{-1} = \varphi(y) y^{-1}. \]
This means that $\varphi(x) x^{-1}$ is constant for any $x\in \Sf_n\setminus \Af_n$. 
Thus, by letting $a=\varphi (x_0)\cdot x_0^{-1} \in \Af_n$ for some $x_0 \in \Sf_n \setminus \Af_n$, we have $\varphi (x)=a x$ for any $x \in \Sf_n \setminus \Af_n$. 

Fix $x \in \Sf_n\setminus \Af_n$ and we take $y \in \Af_n$ arbitrarily. Note $yx \in \Sf_n\setminus \Af_n$. 
Then we see that $\varphi(yx)=ayx$. On the other hand, we also have $\varphi(yx)=\varphi(y)\varphi(x)=yax$. 
Hence, $a \in Z(\Af_n)$. It is well known that $Z(\Af_n)=\{e\}$. Therefore, $a=e$. 
This means that $\varphi=\id$, as required. 
\end{proof}


\subsection{More properties for $P$ and $P^2$}
\label{sec:PandP2}

\begin{lem}\label{lem:P2_1}
Let $G$ and $G'$ be finite groups and let $\psi \in \gauto(G)$ and $\psi' \in \gauto(G')$. 
Assume that there is a group isomorphism $h : P \rightarrow P'$ satisfying $h \circ \psi|_P = \psi'|_{P'} \circ h$. 
Then $h|_{P^2}$ gives a group isomorphism between $P^2$ and $P'^2$, i.e., $P^2 \cong_\mathrm{gr} P'^2$. 
\end{lem}
\begin{proof}
Let $h$ be a group isomorphism satisfying $h \circ \psi|_P = \psi'|_{P'} \circ h$. Then 
\begin{align*}
h(s_a(b))&=h(a \psi(a^{-1}b))=h(a) \cdot h \circ \psi(a^{-1}b)=h(a) \cdot \psi' \circ h(a^{-1}b)=h(a)\psi'(h(a)^{-1}h(b)) \\
&=s_{h(a)}'(h(b))
\end{align*}
for $a,b\in P$.
Hence, $h$ is a quandle automorphism between $P$ and $P'$. 
Therefore, by applying Theorem~\ref{breakthrough} (iii) for $Q(P,\psi|_P)$ and $Q(P',\psi|_{P'})$, we conclude the assertion. 
\end{proof}
\begin{lem}\label{lem:PcapF}
Work with the same notation as in Lemma~\ref{lem:P2_1}. 
Then $|P \cap \fix(\psi,G)|=|P' \cap \fix(\psi',G')|$ holds. 
\end{lem}
\begin{proof}
We see this assertion as follows: 
\begin{align*}
x \in P \cap \fix(\psi,G) \;&\Longleftrightarrow\; x \in P \text{ and }x=\psi(x) \;\Longleftrightarrow\; h(x) \in P' \text{ and }h(x) = h \circ \psi(x) \\
&\Longleftrightarrow\; h(x) \in P' \text{ and }h(x) = \psi' \circ h(x) \;\Longleftrightarrow\; h(x) \in P' \cap \fix(\psi',G'). 
\end{align*}
\end{proof}

For a finite group $G$ and its subset $H$ and $\psi \in \gauto(G)$, we define the \textit{twisted normalizer} of $H$ with respect to $\psi$ as follows: 
$$\TN^\psi_G(H):=\{x \in G : x H \psi(x)^{-1}=H\}.$$
We see that $\TN^\psi_G(H)$ is a subgroup of $G$. In fact, 
\begin{itemize}
\item $e \in \TN^\psi_G(H)$ since $e H \psi(e)^{-1}=H$; 
\item for $x,y \in \TN^\psi_G(H)$, we have $xyH\psi(xy)^{-1}=x(yH\psi(y)^{-1}))\psi(x)^{-1} = H$; 
\item for $x \in G$, we see that $x H \psi(x)^{-1}=H$ holds if and only if $x^{-1}H \psi(x)=H$ holds.
\end{itemize}

For a while, let $F=\fix(\psi,G)$. 
Since $P$ is a normal subgroup of $G$ (by Proposition~\ref{normal_subgroup}), we have $FP=PF$. 
\begin{lem}\label{hodai:normal_TN}
$P$ is a normal subgroup of $\TN^\psi_G(P^2)$. 
\end{lem}
\begin{proof}
\noindent\underline{$P < \TN^\psi_G(P^2)$}:  
Since $P \subset PF$ is trivial, for the proof that $P$ is a subgroup of $\TN^\psi_G(P^2)$, 
it is enough to show the inclusion $PF \subset \TN^\psi_G(P^2)$.  

For any $p_1,p_2 \in P$, any $x \in F$ and any $y\in G$, we know that 
\begin{align*}
s_{p_1xp_2}(y )&=s_{p_1\tilde{p_2}x}(y ) \;\;\text{(for some $\tilde{p_2} \in P$ since $P \triangleleft G$)} \\
&=p_1\tilde{p_2}x\psi(x^{-1}\tilde{p_2}^{-1}p_1^{-1} y ) =p_1\tilde{p_2}\psi(\tilde{p_2}^{-1}p_1^{-1} y) \;\;\text{(since }x \in F) \\
&=s_{p_1\tilde{p_2}}(y ).     
\end{align*}
Fix $p \in P$ and $x \in F$. 
For each $a \in P$, let $\tilde{a}$ denote the element of $P$ satisfying $s_{pxa}=s_{p\tilde{a}}$.

Take $p \in P$ and $x \in F$ arbitrarily. 
By Lemma~\ref{lem_Lx}, for an arbitrary $q \in P^2$, write $L_q=s_{p_1} \circ \cdots \circ s_{p_r}$, where $p_i \in P$.
Then we see that 
\begin{align*}
(px)q\psi(px)^{-1}&=L_{px} \circ L_q (\psi(px)^{-1})=
L_{px}\circ s_{p_1} \circ \cdots \circ s_{p_r}(\psi(px)^{-1})\\
&=s_{pxp_1} \circ \cdots \circ s_{pxp_r}(px \psi(px)^{-1} )
=s_{p \tilde{p_1}} \circ \cdots \circ s_{p \tilde{p}_r}( p \psi(p^{-1}) )\\
&=s_{p \tilde{p_1}} \circ \cdots \circ s_{p \tilde{p}_r}\circ s_p(e) \in P^2. 
\end{align*}
Hence $(px) P^2 \psi(px)^{-1} \subset P^2$ holds. Since $|(px) P^2 \psi(px)^{-1}|=| P^2|$, this implies that $(px) P^2 \psi(px)^{-1} = P^2$.
Therefore, one has $px \in \TN^\psi_G(P^2)$. 

\noindent\underline{$P \triangleleft \TN^\psi_G(P^2)$}:  
Take $p \in P$ and $x \in \TN^\psi_G(P^2)$ arbitrarily. Since $P$ is a normal subgroup of $G$ and $x \in \TN^\psi_G(P^2) \subset G$, 
we see that $xpx^{-1} \in P$, as required. 
\end{proof}

\begin{lem}
Let $Q=Q(G,\psi)$ and $Q'=Q(G',\psi')$ be two generalized Alexander quandles and assume that $Q \cong_\mathrm{qu} Q'$. 
Then $|\TN^\psi_G(P^2)|=|\TN^{\psi'}_{G'}({P'}^2)|$. 
\end{lem}
\begin{proof}
Let $f : Q \rightarrow Q'$ be a quandle automorphism with $f(e)=e'$. 
Then, by Theorem~\ref{breakthrough}, we know that $f|_P$ is a group homomorphism between $P$ and $P'$. 

Once we prove that $f(x) \in \TN^{\psi'}_{G'}(P'^2)$ for any $x \in \TN^\psi_G(P^2)$, since $f$ is injective, we obtain that $|\TN^\psi_G(P^2)| \leq |\TN^{\psi'}_{G'}(P'^2)|$. 
By applying the same discussion, we also obtain the other inequality, as required. 

Let $x \in \TN^\psi_G(P^2)$. For a given $q' \in P'^2$, 
by Lemma~\ref{lem_Lx}, let $L_{q'}=s_{p_1'}' \circ \cdots \circ s_{p_r'}'$, where $p_i' \in P'$ and $r \equiv 0 \pmod{\ord_{\gauto(G')} \psi'}$. 
Then we see the following: 
\begin{align*}
f(x)q'\psi'(f(x))^{-1}
&=L_{f(x)}\circ L_{q'}(\psi'(f(x))^{-1})
=L_{f(x)}\circ s_{p_1'}' \circ \cdots \circ s_{p_r'}'(\psi'(f(x))^{-1}) \\ 
&=s_{f(x)p_1'}' \circ \cdots \circ s_{f(x)p_r'}'(f(x)\psi'(f(x))^{-1})
=s_{f(x)p_1'}' \circ \cdots \circ s_{f(x)p_r'}'\circ s'_{f(x)}(e')\\
&=f\left(s_{xp_1} \circ \cdots \circ s_{xp_r}\circ s_x (e)\right) \\
&\text{(where we let $f(xp_i)=f(x)p_i'$ for each $i$ with $p_i \in P$, see Theorem~\ref{breakthrough} (iv))} \\
&=f\left(L_x \circ s_{p_1} \circ \cdots \circ s_{p_r} \circ s_e(x^{-1}) \right)
=f\left(x\cdot  s_{p_1} \circ \cdots \circ s_{p_r}(e) \cdot s_e(x^{-1}) \right) \\
&=f(x \cdot \underbrace{s_{p_1} \circ \cdots \circ s_{p_r}(e)}_{\in P^2} \cdot \psi(x)^{-1}) \in f(P^2)=P'^2 \;\;\text{(since $x \in \TN^\psi_G(P^2)$)}. 
\end{align*}
This implies that $f(x) \in \TN^{\psi'}_{G'}(P'^2)$. 
\end{proof}

We introduce the following two conditions \ref{G-normal} and \ref{iikanji}. 
\begin{enumerate}[label=(P\arabic*)]
    \item \label{G-normal} $P^2$ is a normal subgroup of $G$.
    \item \label{iikanji} $P^2=\{s_p(e) : p \in P\}$ holds.
\end{enumerate}
\begin{rem}
(i)
For \ref{G-normal}, note that $H \triangleleft G$ and $K \triangleleft H$ do not imply $K \triangleleft G$ in general. 
For \ref{iikanji}, note that the inclusion $\supset$ is always true, but $\subset$ does not necessarily hold.

(ii)
The conditions \ref{G-normal} and \ref{iikanji} are independent.
In fact, $Q(Q_8,\psi_4)$ in Table~\ref{tab:Q8} in Section~\ref{sec:8}
satisfies \ref{G-normal} and does not satisfy \ref{iikanji}.
On the other hand, the following $Q'$ does not satisfy \ref{G-normal} and satisfies \ref{iikanji};
$Q'=Q(G',\psi')$ with $G'=\Sf_3\times \Sf_3$ and $\psi'\colon G' \to G'$ defined by $(x,y)\mapsto (y,x)$.
\end{rem}

The following proposition claims that \ref{G-normal} and \ref{iikanji} are quandle invariants. 

\begin{prop}
Let $Q=Q(G,\psi)$ and $Q'=Q(G',\psi')$ be two generalized Alexander quandles and assume that $Q \cong_\mathrm{qu} Q'$. 
\begin{itemize}
       \item[{\em (i)}] $Q$ satisfies \ref{G-normal} if and only if $Q'$ satisfies \ref{G-normal}. 
       \item[{\em (ii)}] $Q$ satisfies \ref{iikanji} if and only if $Q'$ satisfies \ref{iikanji}. 
\end{itemize}        
\end{prop}
\begin{proof}
Take a quandle isomorphism $f\colon Q\to Q'$ with $f(e)=e'$.

(i) Assume that $Q$ satisfies \ref{G-normal}. For $a'\in G'$ and $q'\in P'^2$, we prove $a' q' a'^{-1} \in P'^2$.
Then there exists $a\in G$ such that $f(a)=a'$.
By Lemma~\ref{lem_Lx}, 
let $L_q'=s'_{p'_1}\circ \cdots \circ s'_{p'_r} $,
where $p'_i\in P'$ and $r \equiv 0 \pmod{\ord_{\gauto(G')} \psi'}$. 
Then we have
    \begin{align*}
        a' q' a'^{-1}  &= f(a) q' f(a)^{-1}=L_{f(a)}\circ L_{q'}(f(a)^{-1})\\
        &= L_{f(a)} \circ s'_{p'_1}\circ \cdots \circ s'_{p'_r} ( f(a)^{-1})\\
        &= s'_{f(a) p'_1}\circ \cdots \circ s'_{f(a) p'_r} (e') \\
        &= s'_{f(a p_1)}\circ \cdots \circ s'_{f(a p_r)} (e')\quad (\text{by Theorem~\ref{breakthrough} (iv)})\\
        &= f(s_{a p_1}\circ \cdots \circ s_{a p_r} (e) )\\
        &= f(a \cdot s_{p_1}\circ \cdots \circ s_{p_r} (e) \cdot a^{-1}). 
    \end{align*} 
    Since $Q$ satisfies \ref{G-normal}, we have $a \cdot s_{p_1}\circ \cdots \circ s_{p_r} (e) \cdot a^{-1} \in P^2$.
    By Lemma~\ref{lem:P2_1}, we have $a' q' a'^{-1}\in P'^2$, i.e., $Q'$ satisfies \ref{G-normal}.

    (ii) Assume that $Q$ satisfies \ref{iikanji}. For $q'\in P'^2$, we prove that there exists $p'\in P'$ such that $q'=s'_{p'}(e')$.
    By Lemma~\ref{lem:P2_1}, there exists $q\in P^2$ such that $q'=f(q)$. Since $Q$ satisfies \ref{iikanji}, there exists $p\in P$ such that $q=s_p(e)$.
    Then we have $q'=f(s_p (e))=s'_{f(p)} (e')$. By Theorem~\ref{breakthrough} (i), we have $f(p)\in P'$.
    Hence $Q'$ satisfies \ref{iikanji}. 
\end{proof}

\begin{lem}\label{lem:PF}
Assume that $Q$ satisfies \ref{G-normal} and \ref{iikanji}. Then $\TN^\psi_G(P^2)=PF$. 
\end{lem}
\begin{proof}
Since $PF \subset \TN^\psi_G(P^2)$ is always true by the proof of Lemma~\ref{hodai:normal_TN}, we may prove the other inclusion. 

Take $x \in \TN^\psi_G(P^2)$. Then $x P^2 \psi(x)^{-1}=P^2$. 
Since $P^2$ is a normal subgroup of $G$ by \ref{G-normal}, 
we have $P^2\psi^{-1}(x) = \psi^{-1}(x)P^2$. Hence, $x \psi(x)^{-1} \in P^2$, i.e., there is $q \in P^2$ with $x\psi(x)^{-1}=q$.  
By \ref{iikanji}, there is $p \in P$ such that $q=s_p(e)$, i.e., $x\psi(x^{-1})=s_p(e)=p\psi(p^{-1})$. 
Hence, $p^{-1}x=\psi(p^{-1}x)$, i.e., $p^{-1}x \in F$. Therefore, $x \in PF$, as desired. 
\end{proof}

\subsection{Proof of Theorem~\ref{monosugoi}}
Now, we are in the position to give a proof of Theorem~\ref{monosugoi}. 
The whole part of this subsection is devoted to giving its proof. 

\medskip

First, we prove the ``Only if'' part. 
Assume that $Q \cong_\mathrm{qu} Q'$. Then $|G|=|G'|$ trivially follows and $|\fix(\psi,G)|=|\fix(\psi',G')|$ also follows by Remark~\ref{rem:GneqG'}. 
Let $f : Q \rightarrow Q'$ be a quandle automorphism with $f(e)=e'$ and let $h=f|_P$. 
Then $h$ becomes a group homomorphism satisfying $h \circ \psi|_P = \psi'|_{P'} \circ h$ (see Theorem~\ref{breakthrough}) 
and (C-2) easily follows from $f(s_a(e))=f|_P(s_a(e))=h(s_a(e))$ and $f(s_a(e))=s_{f(a)}'(e')$, where $f(a) \in G'$. 

\medskip

In the remaining part, we give the proof of the ``If'' part. 
Let us assume the conditions (A), (B) and (C). 
We define the set $S$ (resp. $S'$) of coset representatives of $P^2$ (resp. $P'^2$) as follows: 
$$S=\{s_a(e)P^2 : a \in G\}\text{ and }S'=\{s_{a'}'(e')P'^2 : a' \in G'\}.$$

\noindent
{\bf The first step}: 
First, we compute $|S|$. For this, let 
$$\pi : G/P \rightarrow S, \;\; \pi(aP)=s_a(e)P^2.$$
The well-definedness is verified as follows: for any $p \in P$ and $q \in P^2$, we see that 
\begin{align*}
s_{ap}(e)q&=ap\psi(ap)^{-1}q=ap\psi(p^{-1})\psi(a)^{-1}q=a\underbrace{s_p(e)}_{\in P^2}\psi(a)^{-1}q \\
&=a\psi(a)^{-1}\tilde{q} q, \;\text{where $\tilde{q} \in P^2$ (by \ref{G-normal})} \\
&\in s_a(e) P^2. 
\end{align*}
Since $\pi$ is surjective, for the computation of $|S|$, we may compute $|\pi^{-1}(s_a(e)P^2)|$. We see the following: 
\begin{align*}
&bP \in \pi^{-1}(s_a(e)P^2) \;\Longleftrightarrow\; \pi(bP) =s_a(e)P^2 \;\Longleftrightarrow\; s_b(e)P^2=s_a(e)P^2  \\
\Longleftrightarrow\;\; &b\psi(b)^{-1}P^2 = a\psi(a)^{-1}P^2 \;\Longleftrightarrow\; bP^2\psi(b)^{-1} = aP^2\psi(a)^{-1} \text{ (by \ref{G-normal})} \\
\Longleftrightarrow\;\; &a^{-1}bP^2\psi(a^{-1}b)^{-1} = P^2 \;\Longleftrightarrow\; a^{-1}b \in \TN^\psi_G(P^2)=PF \text{ (by Lemma~\ref{lem:PF})}\\
\Longleftrightarrow\;\; &a^{-1}bP \subset PF \;\Longleftrightarrow\; bP\subset aPF.
\end{align*}
This implies that $|\pi^{-1}(s_a(e)P^2)|=|PF/P|=|F/P \cap F|$. Hence, 
$$|S|=\frac{|G/P|}{|\pi^{-1}(s_a(e)P^2)|}=\frac{|G/P|}{|F/P \cap F|} = \frac{|G| \cdot |P \cap F|}{|P| \cdot |F|}.$$

\noindent
{\bf The second step}: We construct a bijection between $S$ and $S'$ using $h$. 
Firstly, we define the map $\tilde{h}\colon P/P^2 \to P'/P'^2$ by setting $\tilde{h}(pP^2):=h(p)P'^2$.
Since 
\begin{align*}
    p_1 P^2 =p_2 P^2
    &\;\Longleftrightarrow\;
    p_2^{-1} p_1 \in P^2
    \;\Longleftrightarrow\;
    h(p_2^{-1} p_1) \in P'^2 \quad (\text{by Lemma~\ref{lem:P2_1}})\\
    &\;\Longleftrightarrow\;
    h(p_1)P'^2=h(p_2)P'^2,
\end{align*}
we see that $\tilde{h}$ is well-defined and injective. 
The surjectivity of $\tilde{h}$ follows by the the surjectivity of $h|_{P^2}\colon P^2 \to P'^2$.
Hence, $\tilde{h}$ is a bijection between $P/P^2$ and $P'/P'^2$. 
Next, we can show that the restriction $\tilde{h}|_S$ of $\tilde{h}$ into $S$ gives a bijection between $S$ and $S'$.
\begin{itemize}
\item We have $\tilde{h}(S) \subset S'$ since $\tilde{h}(s_a(e)P^2)=h(s_a(e))P'^2=s_{a'}'(e')P'^2$ by (C-2).
\item Since $\tilde{h}$ is injective, $\tilde{h}|_S$ is also injective.
\item Moreover, we have $|S|=|S'|$. In fact, by our assumptions, we have $|G|=|G'|$, $|F|=|F'|$ and $|P|=|P'|$ and 
$|P \cap F|=|P' \cap F'|$ by Lemma~\ref{lem:PcapF}. Hence, by the first step, we see that 
\begin{align*}
|S|=\frac{|G| \cdot |P \cap F|}{|P| \cdot |F|}=\frac{|G'| \cdot |P' \cap F'|}{|P'| \cdot |F'|}=|S'|. 
\end{align*}
\end{itemize}
Therefore, $\tilde{h}$ gives a bijection between $S$ and $S'$.

\noindent
{\bf The third step}: 
Let $A\subset G$ be a set of coset representatives of $G/P$.
We construct a map $k\colon A\to G'$ such that $k(A)$ is a set of coset representatives of $G'/P'$ and $h(s_a(e))=s'_{k(a)}(e')$ holds for any $a\in A$.

Let $\pi':G'/P' \rightarrow S'$ similarly to $\pi$.
Firstly, we construct a bijective map $k\colon G/P\to G'/P'$ such that $\tilde{h}|_S(\pi(aP))=\pi'(k(aP'))$ holds for any $aP\in G/P$.
\[
  \xymatrix{
    G/P \ar[r]^-{\simeq}_-{k} \ar@{->>}[d]_{\pi} &  G'/P'  \ar@{->>}[d]^{\pi'}\\
     S \ar[r]^-{\simeq}_-{\tilde{h}|_S}  & S'  \ar@{}[lu]|{\circlearrowright}
  }
\]

Since $\displaystyle |\pi^{-1}(\pi(aP))|=\frac{|F|}{|P \cap F|} =\frac{|F'|}{|P' \cap F'|}=|\pi'^{-1}(\tilde{h}(\pi(aP)))|$, 
we can construct a bijection $k_a$ between $\pi^{-1}(\pi(aP))$ and $\pi'^{-1}(\tilde{h}(\pi(aP)))$ for each fixed $aP \in G/P$. 
Thus, we can construct a bijection $k$ between $G/P$ and $G'/P'$ by attaching $k_a$'s for $aP \in G/P$.

Next, we apply $k$ to $A$ using the same symbol.
Then for any $a \in A$, we have $h(s_a(e))P'^2=s'_{k(a)}(e')P'^2$, i.e., $h(s_a(e))=s_{k(a)}'(e')q'$ for some $q' \in P'^2$. 
Here, we see that 
\begin{align*}
s'_{k(a)}(e')q' &= k(a)\psi'(k(a))^{-1}q'=k(a)\tilde{q}'\psi'(k(a))^{-1} \text{ (for some }\tilde{q}'\text{ since $Q'$ satisfies \ref{G-normal})} \\
&=k(a)s_{p'}'(e')\psi'(k(a))^{-1} \text{ (since $Q'$ satisfies \ref{iikanji})} \\
&=k(a)p'\psi'(p')^{-1}\psi'(k(a))^{-1} = k(a)p'\psi'((k(a)p')^{-1}) \\
&=s_{k(a)p'}'(e'). 
\end{align*}
Note that $k(a)$ and $k(a)p'$ belong to the same coset of $G'/P'$.
By replacing $k(a)$ into $k(a)p'$, we can construct the desired map $k:A \rightarrow G'$.

\noindent
{\bf The fourth step}: Now, we can construct a quandle automorphism between $Q$ and $Q'$ 
by using the bijections $h:P \rightarrow P'$ (guaranteed by the condition (C)) and $k:A \rightarrow k(A)$ constructed in the third step as follows.  
For any $x \in G$, we uniquely write $x$ for $x=a_xp_x$, where $a_x \in A$ and $p_x\in P$. 
By $a_xp_xa_x^{-1}\in P$ since $P$ is a normal subgroup of $G$, 
we can define $f(x)$ by setting $h(a_xp_xa_x^{-1})k(a_x)$. The bijectivity of $f$ is trivial since $h$ and $k$ are bijective.
Since
\[
a_x \psi(p_x^{-1} a_x^{-1})=a_x \psi(a_x^{-1})\psi(a_x p_x^{-1} a_x^{-1}) =s_{a_x}(e)\cdot \psi(a_x p_x^{-1} a_x^{-1}) \in P
\]
and the properties of $h$ and $k$, we have 
\begin{align}
    \label{eq:kantan}
    h(a_x \psi(p_x^{-1} a_x^{-1}))
    &=h(s_{a_x}(e) \psi(a_x p_x^{-1} a_x^{-1}))
    =h(s_{a_x}(e)) h(\psi(a_x p_x^{-1} a_x^{-1}))\\
    &=s'_{k(a_x)}(e')  \psi' (h(a_x p_x^{-1} a_x^{-1}))
    =k(a_x)\psi'(k(a_x)^{-1}) \psi' (h(a_x p_x^{-1} a_x^{-1})) \notag\\
    &=k(a_x) \psi'(f(x)^{-1}). \notag
\end{align}
Using \eqref{eq:kantan},
we can check that $f$ is a quandle homomorphism as follows:
For elements $x=a_x p_x$ and $y=a_y p_y$ in $G$,
\begin{align*}
f \circ s_x(y)&=f( s_{a_xp_x}(a_yp_y))=f(a_y \cdot s_{a_y^{-1}a_xp_x}(p_y))\\
&=h(a_y\cdot s_{a_y^{-1}a_xp_x}(p_y)\cdot a_y^{-1}) k(a_y) \quad (\text{since $s_{a_y^{-1}a_xp_x}(p_y)\in P$})\\
&=h(a_y\cdot a_y^{-1}a_xp_x \psi(p_x^{-1} a_x^{-1} a_y p_y)\cdot a_y^{-1}) k(a_y) \\
&=h(a_x p_x a_x^{-1} \cdot a_x \psi(p_x^{-1} a_x^{-1}) \cdot \psi(a_y p_y) a_y^{-1}) k(a_y) \\
&=h(a_x p_x a_x^{-1}) \cdot h(a_x \psi(p_x^{-1} a_x^{-1})) \cdot h(a_y \psi(p_y^{-1} a_y^{-1}))^{-1} k(a_y) \\
&=h(a_x p_x a_x^{-1}) \cdot k(a_x) \psi'(f(x)^{-1}) \cdot \psi'(f(y)) k(a_y)^{-1} \cdot k(a_y)
\quad(\text{by \eqref{eq:kantan}})\\
&=f(x) \psi'(f(x)^{-1}f(y))\\
&=s'_{f(x)} (f(y)),
\end{align*}
as desired.

\bigskip

\section{Application of Theorem~\ref{monosugoi} to dihedral groups}\label{sec:D_n}

In this section, we discuss Problem~\ref{toi} for dihedral groups. 

For this, we first collect some fundamental materials on dihedral groups and their automorphisms. 
Let $D_n$ be the dihedral group of order $2n$. We use the notation $\sigma$ and $\tau$ as follows: 
$$D_n=\{ \sigma^i : i=0,1,\ldots,n-1 \} \cup \{\tau \sigma^i : i=0,1,\ldots,n-1 \}$$
equipped with the relations $\tau^2=\sigma^n=e$ and $\sigma\tau=\tau\sigma^{-1}$. 

\bigskip

We recall the structure of $\gauto(D_n)$. 
Let $$\mathrm{Aff}(C_n)=\{\varphi_{a,b} : a \in C_n^\times, b \in C_n\},$$ 
where for $\tau^\epsilon \sigma^i \in D_n$ with $\epsilon \in \{0,1\}$ and $0 \leq i \leq n-1$, 
we define a group homomorphism $\varphi_{a,b}:D_n \rightarrow D_n$ by setting 
$$\varphi_{a,b}(\tau^\epsilon \sigma^i)=\tau^\epsilon \sigma^{ai+\epsilon b}.$$ 
Then it is well known that $$\gauto(D_n) \cong_\mathrm{gr} \mathrm{Aff}(C_n).$$ 
Note that $\varphi_{c,d}^{-1}=\varphi_{c^{-1},-c^{-1}d}$ and 
\begin{align}\label{eq:conj_D_n}
\varphi_{c,d}\circ \varphi_{a,b} \circ \varphi_{c,d}^{-1}=\varphi_{a,b c+(1-a)d} 
\end{align}
hold. 

In what follows, for $a \in C_n^\times$ and $b \in C_n$, we use the notation $Q(D_n,a,b)$ instead of $Q(D_n,\varphi_{a,b})$.


\subsection{Conjugacy classes for $\gauto(D_n)$}
We discuss the conjugacy classes of $\gauto(D_n)$. 

\begin{lem}\label{hodai_D_n}
Let $c,d$ are integers and $m,n$ are positive integers.
\begin{itemize}
    \item[\em (i)] A linear congruence $c z \equiv d \pmod{n}$ with the variable $z$ has a solution
    if and only if $d$ is a multiple of $\gcd(n,c)$.
    Moreover if this linear congruence has a solution $z_0$,
    then the solutions $z$ are of the form $z\equiv z_0 + \frac{n}{\gcd(n,c)}i$ $(i=0,1,\ldots ,\gcd(n,c)-1)$.
    \item[\em (ii)] There exists $p$ with $\gcd(n,p)=1$ such that $pc \equiv \gcd(m,c) \pmod m$.
    \item[\em (iii)] We have $$\{cb : b \in C_n\}=\{gy : y \in C_n\},$$ 
    where $g=\gcd(n,c)$. 
\end{itemize}
\end{lem}
\begin{proof}
(i)
This is well known (see, e.g.,~\cite[Theorem~2.2]{Nathanson}).

\noindent
(ii) 
Let $g=\gcd(m,c)$. Then there are $x,y \in \ZZ$ such that $mx+cy=g$. Hence, $cy \equiv g \pmod m$. 
Let $m'=m/g$ and $c'=c/g$. Then $m'x+c'y=1$. Note that $\gcd(m',y)=1$. 
Thus, by Dirichlet's theorem on arithmetic progressions, there are infinitely many primes $p$ which is of the form $p=m'k+y$ ($k \in \ZZ$). 
For each $p$, we have $$pc=(m'k+y)c'g=mka'+cy \equiv g \pmod m.$$ 
Since there are infinitely many such primes $p$, we can choose a prime $p$ with $\gcd(n,p)=1$, as required. 

\noindent
(iii) One inclusion is trivial since $g$ divides $c$. Thus it is enough to show that $g$ is contained in $\{cb : b \in C_n\}$. 
By (i), there is $b \in C_n$ such that $cb \equiv g \pmod n$. 
Hence, one has $g \in \{cb : b \in C_n\}$, as required. 
\end{proof}

\begin{prop}\label{prop:conj_class_D_n}
Let $\varphi_{a,b},\varphi_{a',b'}\in \gauto(D_n)$
and put $d=\gcd(n,1-a,b)$ and $d'=\gcd(n,1-a',b')$.
Then $\varphi_{a,b}$ and $\varphi_{a',b'}$ are conjugate if and only if the following two conditions are satisfied:
\begin{itemize}
    \item $a\equiv a' \pmod{n}$;
    \item $d=d'$,
\end{itemize}
i.e., the set of representatives of conjugacy classes of $\gauto(D_n)$ is given by 
\begin{equation}\label{eq:set_D_n}\begin{split}
\{\varphi_{a,d} : a \in C_n^\times, \; d \text{ is a divisor of } g\}, 
\end{split}\end{equation}
where $g=\gcd(n,1-a)$. 
\end{prop}
\begin{proof}
Let $g=\gcd(n,1-a)$ and $g'=\gcd(n,1-a')$. Note that $d=\gcd(g,b)$ 
since $\gcd(\alpha,\beta,\gamma)=\gcd(\gcd(\alpha,\beta),\gamma)$ holds for each $\alpha,\beta,\gamma \in \ZZ$.

Assume that $\varphi_{a,b}$ and $\varphi_{a',b'}$ are conjugate.
By \eqref{eq:conj_D_n}, there exist $r\in C_n^\times$ and $s\in C_n$
such that $\varphi_{a',b'}=\varphi_{a,b r+(1-a)s}$.
Then $a\equiv a' \pmod{n}$ trivially holds.
Moreover we have
\[
    d'=\gcd(g',b')
    =\gcd(g,b r+(1-a)s)
    =\gcd(g,b r)
    =\gcd(g,b)=d
\]
since $(1-a)s$ is a multiple of $g$ and $r\in C_n^\times$.

Assume that $a\equiv a' \pmod{n}$ and $d=d'$ hold.
In order to show that $\varphi_{a,b}$ and $\varphi_{a',b'}$ are conjugate,
we find $r\in C_n^\times$ and $s\in C_n$ such that $b'\equiv b r+(1-a)s \pmod{n}$.
By Lemma~\ref{hodai_D_n}~(ii), there exist $p,q\in C_n^\times$ such that $p b \equiv d,\  q b' \equiv d \pmod{g}$.
Hence there exist $k,k'\in \ZZ$ such that $d\equiv p b + k g \pmod{n}$ and $b'\equiv q^{-1} d + k' g \pmod{n}$.
On the other hand, by Lemma~\ref{hodai_D_n}~(i), there exists $z_0\in C_n$ such that $(1-a) z_0 \equiv g \pmod{n}$.
Therefore we have
\[
b'\equiv q^{-1} d + k' g
\equiv q^{-1} (p b + k g) + k' g
\equiv q^{-1} p b + (q^{-1} k+k')z_0 (1-a) \pmod{n}.
\]
Then $q^{-1} p \in C_n^\times$ and $(q^{-1} k+k')z_0 \in C_n$ are the desired $r$ and $s$, respectively.
\end{proof}

Note that $\varphi_{a,0}$ is conjugate to $\varphi_{a,g}$ in $\gauto(D_n)$, where $a \in C_n^\times$ and $g=\gcd(n,1-a)$. 
In what follows, we will use $\varphi_{a,g}$ as a representative of conjugacy classes of $\gauto(D_n)$ instead of $\varphi_{a,0}$, 
but we will use $\varphi_{a,0}$ for the computation.

\begin{ex}\label{ex:cong}
A set of representatives of conjugacy classes of each of $\gauto(D_4)$, $\gauto(D_5)$, $\gauto(D_6)$ and $\gauto(D_8)$ is given as follows: 
\begin{align*}
\gauto(D_4): \;\; &\{\varphi_{1,0},\varphi_{1,1},\varphi_{1,2},\varphi_{3,1},\varphi_{3,2} \} \\
\gauto(D_5): \;\; &\{\varphi_{1,0},\varphi_{1,1},\varphi_{2,1},\varphi_{3,1},\varphi_{4,1} \} \\
\gauto(D_6): \;\; &\{\varphi_{1,i} : i=0,1,2,3 \} \cup \{\varphi_{5,i} : i=1,2 \} \\
\gauto(D_8): \;\; &\{\varphi_{1,i} : i=0,1,2,4 \} \cup \{\varphi_{3,i},\varphi_{7,i} : i=1,2 \} \cup \{\varphi_{5,i} : i=1,2,4\} \\
\end{align*}
\end{ex}


\subsection{Invariants for $Q(D_n,a,b)$}
\begin{prop}\label{prop:D_n_Fix}
Let $a \in C_n^\times$, $b \in C_n$, $g=\gcd(n,1-a)$ and $d=\gcd(g,b)$. Then 
\begin{align*}
\fix(\varphi_{a,b},D_n)=\begin{cases}
\langle \sigma^{n/g}, \; \tau \sigma^{z_0} \rangle\ \text{for some $z_0$}, \quad &\text{if }d =g, \\
\langle \sigma^{n/g} \rangle, \quad &\text{if }d \neq g.  
\end{cases}
\end{align*}
In particular, $|\fix(\varphi_{a,b},D_n)|=\begin{cases}2g, &\text{if }d =g, \\ g, &\text{if }d \neq g. \end{cases}$
\end{prop}
\begin{proof}
By $\varphi_{a,b}(\sigma^i)=\sigma^{ai}$, we see that $\varphi_{a,b}(\sigma^i)=\sigma^i$ if and only if $(1-a)i \equiv 0$ (mod $n$). 
By Lemma~\ref{hodai_D_n}~(i), $i$ is a multiple of $n/g$.

Moreover, by $\varphi_{a,b}(\tau\sigma^i)=\tau\sigma^{ai+b}$, 
we also see that $\varphi_{a,b}(\tau\sigma^i)=\tau\sigma^i$ if and only if $(1-a)i \equiv b \pmod{n}$.
By Lemma~\ref{hodai_D_n}~(i), this has a solution if and only if $b$ is a multiple of $g$, i.e., $d=g$.
Moreover, when $d=g$, the solutions are $i\equiv z_0 +\frac{n}{g}j \pmod{n}$ ($j=0,1,\ldots g-1$).
\end{proof}

For $x, y \in Q(D_n,a,b)$, let $x=\tau^\epsilon \sigma^i$ and $y=\tau^\delta \sigma^j$. 
Then we see the following: 
$$
s_x(y)=\tau^\delta \sigma^{(-1)^{\epsilon+\delta}(1-a)i+aj+\mu(\epsilon+\delta)b},
$$
where
\begin{equation}
\label{eq:mu-def}
\mu(i)=
\begin{cases}
1 & \text{if $i$ is odd},\\
0 & \text{if $i$ is even}\\
\end{cases}         
\end{equation}
for $i\in C_{2}$.
In particular, we have 
\begin{equation}\label{eq:sxe}
     s_x(e)=\sigma^{(-1)^\epsilon(1 - a) i+\epsilon b}.
\end{equation}

\begin{prop}\label{prop:p}
Let $a \in C_n^\times$, $b \in C_n$, $g=\gcd(n,1-a)$, $d=\gcd(g,b)$ and $g_2=\gcd(n/d,1-a)$. 
Then $$P(Q(D_n,a,b))=\langle \sigma^{d} \rangle \;\;\text{and}\;\; P^2(Q(D_n,a,b))=\langle \sigma^{dg_2} \rangle.$$ 
\end{prop}
\begin{proof}
It directly follows from \eqref{eq:sxe} and Lemma~\ref{hodai_D_n} (iii) that $P(Q(D_n,a,b))=\langle \sigma^{d} \rangle$. 

For $p \in P(Q(D_n,a,b))$, let $p=\sigma^{dj}$ for some $j$. Then $s_p(e)=\sigma^{d(1-a)j}$. 
Since $j$ runs over $j=0,1,\ldots,n/d-1$, we obtain that $P^2(Q(D_n,a,b))=\langle \sigma^{dg_2} \rangle$. 
\end{proof}

\subsection{Proof of Corollary~\ref{cor2}}
Now, we give a proof of Corollary~\ref{cor2}. 

\noindent
{\bf The first step}: First, we check that $Q(D_n,a,b)$ satisfies \ref{G-normal} and \ref{iikanji}.

For any $\tau^\epsilon \sigma^i \in D_n$ and any $\sigma^{dg_2j} \in P^2$, we see the following: 
$$
(\tau^\epsilon \sigma^i) \cdot \sigma^{dg_2j} \cdot (\tau^\epsilon \sigma^i)^{-1} 
=\tau^\epsilon \sigma^{i+dg_2j-i} \tau^\epsilon = \sigma^{(-1)^\epsilon dg_2j} \in P^2, 
$$
which implies that $Q(D_n,a,b)$ satisfies \ref{G-normal}. 

Fix $\sigma^{d g_2 j} \in P^2$. Since $\gcd(n,(1-a)d)=dg_2$, by Lemma~\ref{hodai_D_n}~(i),
there exists $z_0\in C_n$ such that $(1-a) d z_0 \equiv d g_2 j \pmod{n}$.
Hence, for any $\sigma^{dg_2j} \in P^2$, by letting $p=\sigma^{d z_0} \in P$, we see that 
$$s_p(e)=\sigma^{(1-a)d z_0}=\sigma^{dg_2j}.$$
Thus, $Q(D_n,a,b)$ satisfies \ref{iikanji}. 

Therefore, we can apply Theorem~\ref{monosugoi} for $G=G'=D_n$. 

\noindent
{\bf The second step}: 
Next, we prove the ``Only if'' part of the statement. 
Let $Q=Q(D_n,a,b)$ and $Q'=Q(D_n,a',b')$ and assume that $Q \cong_\mathrm{qu} Q'$. 
\begin{itemize}
\item Then $|\fix(\varphi_{a,b},D_n)|=|\fix(\varphi_{a',b'},D_n)|$ holds. 
\item By Proposition~\ref{prop:p}, we see that $P=P(Q) \cong_\mathrm{gr} C_{n/d}$ and $P'=P(Q')  \cong_\mathrm{gr} C_{n/d'}$. Thus, we obtain that $d=d'$. 
\item Let $h : C_{n/d} \rightarrow C_{n/d}$ be a group isomorphism satisfying the condition (C) in Theorem~\ref{monosugoi}. 
Note that $\varphi_{a,b}|_P$, $\varphi_{a',b'}|_{P'}$ and $h$ are multiplication maps by some $a,a',c \in C_{n/d}^\times$, respectively.
By (C-1),  we have $ca=a'c$ in $C_{n/d}$. Since $c \in C_{n/d}^\times$, we have $a=a'$ in $C_{n/d}$.
\end{itemize}

\noindent
{\bf The third step}: 
Finally, we prove the ``If'' part of the statement. Let $a,a' \in C_n^\times$ and $b,b' \in C_n$ which satisfy all the conditions in Corollary~\ref{cor2}. 
Then the conditions (A) and (B) in Theorem~\ref{monosugoi} are clearly satisfied. 
Let us consider the identity map $h=\id_{C_{n/d}}$. Then $h : P \rightarrow P'$ is a group isomorphism. 
In what follows, we verify that $h$ satisfies (C-1) and (C-2). 
\begin{itemize}
\item[(C-1)] This is equivalent to the condition $a \equiv a' \pmod{n/d}$ as discussed in the second step. 
\item[(C-2)] We divide the discussions into two cases: 
\begin{itemize}
\item ``$d=g$ and $d'=g'$'' or ``$d \neq g$ and $d' \neq g'$''; 
\item $d \neq g$ and $d' = g'$ (the case $d=g$ and $d' \neq g'$ is similar). 
\end{itemize}

For the first case, it follows from the assumption $|\fix(\varphi_{a,b},D_n)|=|\fix(\varphi_{a',b'},D_n)|$ together with Proposition~\ref{prop:D_n_Fix} that $g=g'$ holds.
By Lemma~\ref{hodai_D_n}~(i), there exists $z_0\in C_n$ such that $(1-a')z_0\equiv 1-a \pmod{n}$.
Then, for any $x=\tau^\epsilon \sigma^i$, by letting $x'=\tau^\epsilon \sigma^{z_0i}$, we see that 
\begin{align*}
h(s_x(e))&=s_x(e)=\sigma^{(-1)^\epsilon(1 - a) i+\epsilon b}
=\sigma^{(-1)^\epsilon(1 - a') z_0i+\epsilon b}=s_{x'}'(e). 
\end{align*}
See \eqref{eq:sxe}. 

For the second case, we have $g=2g'$ (see Proposition~\ref{prop:D_n_Fix}). 
By Lemma~\ref{hodai_D_n}~(i), there exists $z_0\in C_n$ such that $(1-a')z_0\equiv 1-a \pmod{n}$.
Similar to the above, we obtain that $h(s_x(e))=s_{x'}'(e)$. 
\end{itemize}

\bigskip


\section{Relationships among generalized Alexander quandles arising from $C_{2n}$ and $D_n$}\label{sec:relation}

In this section, we first discuss Problem~\ref{toi} for finite abelian groups. 
We remark that a solution of Problem~\ref{toi} for finite abelian groups was already given in \cite[Theorem 2.1]{Nelson}, 
but this paper gives another interpretation of \cite[Theorem 2.1]{Nelson} from viewpoints of group theory. 
After those discussions, we investigate a relationship between $Q(D_n,a,b)$ and $Q(C_{2n},a')$, where $a,a' \in C_n^\times$ and $b \in C_n$.

\subsection{Proof of Corollary~\ref{cor1}}
\label{sec:abel}
\begin{proof}[Proof of Corollary~\ref{cor1}]
We apply Theorem~\ref{monosugoi} to finite abelian groups. 
What we may prove is the following: 
\begin{itemize}
\item[(i)] $P=\Im \rho$; 
\item[(ii)] $Q(G,\psi)$ always satisfies \ref{G-normal} and \ref{iikanji} when $G$ is abelian; 
\item[(iii)] the condition (C-2) is automatically satisfied when $G$ is abelian; 
\item[(iv)] the condition (B) follows from (A) and (C).  
\end{itemize}

\noindent (i)
Since $G$ is abelian, we see that 
\[
P
=\{x + \psi(-x) : x \in G\}
=\{x - \psi(x) : x \in G\}
=\Im (\mathrm{id}_G - \psi)
=\Im \rho.        
\]

\noindent (ii)
\ref{G-normal} is trivial since $P^2$ is a subgroup of an abelian group $G$, and thus, $P^2$ is also abelian. 
Moreover, since $P^2=\Im \rho^2$ holds, we see that $Q(G,\psi)$ satisfies \ref{iikanji}. 

\noindent (iii)
Since $h(s_a(e)) \in \Im \rho'=P'$ for each $a \in G$, we see that there is $a' \in G'$ with $h(s_a(e))=\rho'(a')=s_{a'}(e')$. 

\noindent (iv)
We know that $\fix(\psi,G)=\{x \in G : \psi(x)=x\} = \Ker \rho$. Hence, $\Im \rho \cong_\mathrm{gr} G/\Ker \rho$. 
Since $|G|=|G'|$ by (A) and $|P|=|P'|$ by (C), we conclude that 
$$|\fix(\psi,G)|=\frac{|G|}{|\Im \rho|}=\frac{|G|}{|P|}=\frac{|G'|}{|P'|}=|\fix(\psi',G')|,$$
as required. 
\end{proof}
\begin{rem}
We see that \cite[Theorem 2.1]{Nelson} can be rephrased as Corollary~\ref{cor1}, which is explained as follows.  

Let us recall the statement of \cite[Theorem 2.1]{Nelson}: 
two finite Alexander quandles $M$ and $N$ of the same cardinality are isomorphic as quandles 
if and only if there is an isomorphism of $\Lambda$-modules $h:(1-t)M \rightarrow (1-t)N$. (Please refer \cite[Section 1]{Nelson} for undefined notions.) 
We see that $(1-t)M$ (resp. $(1-t)N$) is isomorphic to $\Im \rho$ (resp. $\Im \rho'$) as abelian groups, 
and we can interpret the existence of an isomorphism $h:(1-t)M \rightarrow (1-t)N$ of $\Lambda$-modules 
as the existence of a group isomorphism $h: \Im \rho \rightarrow \Im \rho'$ satisfying $h \circ \psi|_{\Im \rho} = \psi'|_{\Im \rho'} \circ h$. 
\end{rem}

We also give a proof of \cite[Corollary 2.2]{Nelson} by using Corollary~\ref{cor1}. 
\begin{cor}[{\cite[Corollary 2.2]{Nelson}}]\label{cor:Nelson}
Let $a,a' \in C_n^\times$ and $g=\gcd(n,1-a)$ and $g'=\gcd(n,1-a')$. 
Let $Q=Q(C_n,a)$ and let $Q'=Q(C_n,a')$. Then $Q \cong_\mathrm{qu} Q'$ if and only if $g=g'$ and $a \equiv a' \pmod{n/g}$. 
\end{cor}
\begin{proof}
Note that $\gauto(C_n) \cong_\mathrm{gr} C_n^\times$ as groups 
and $\Im \rho$ in the statement of Corollary~\ref{cor1} is equal to $\{g y : y \in C_n\} =\langle g \rangle \cong_\mathrm{gr} C_{n/g}$
by Lemma~\ref{hodai_D_n} (iii). 
\begin{itemize}
\item The condition on the existence of a group isomorphism $h : \Im \rho \rightarrow \Im \rho'$ is equivalent to $n/g = n/g'$, i.e., $g=g'$. 
Then $h$ corresponds to some element of $C_{n/g}^\times$. 
\item The condition $h \circ \psi|_{\Im \rho} = \psi'|_{\Im \rho'} \circ h$ is equivalent to that $ha \equiv a'h \pmod{n/g}$, i.e., $a \equiv a' \pmod{n/g}$. 
\end{itemize}
\end{proof}

\subsection{Proof of Corollary~\ref{cor3}}
\begin{proof}[Proof of Corollary~\ref{cor3}]
Let $Q=Q(C_{2n},a)$ and let $Q'=Q(D_n,\tilde{a},g)$. We have already discussed that 
$Q$ and $Q'$ satisfy \ref{G-normal} and \ref{iikanji} in Proof of Corollary~\ref{cor1} and Section~\ref{sec:D_n}, respectively. 
Thus, we can apply Theorem~\ref{monosugoi} and it is enough to show that $Q$ and $Q'$ satisfy the conditions (A),(B) and (C). 

On (A), it is trivial since $|C_{2n}|=|D_n|=2n$. 
On (B), we have $|\fix(a,C_{2n})|=|\{x \in C_{2n} : (1-a)x=0\}|=\gcd(2n,1-a)=2\gcd(n,k)=2g$. 
On the other hand, we see that $d'=\gcd(n,1-\tilde{a},g)=\gcd(n,2k,g)=g$ and $g'=\gcd(n,1-\tilde{a})=\gcd(n,-2k)$ is equal to $g$ or $2g$.
If $g'=g$, then $d'$ coincides with $g'$. Hence $|\fix(\varphi_{\tilde{a},g},D_n)|=2g'=2g$.
If $g'=2g$, then $d'$ is a proper divisor of $g'$. Hence $|\fix(\varphi_{\tilde{a},g},D_n)|=g'=2g$.
Thus, in all cases, we obtain $|\fix(\varphi_{\tilde{a},g},D_n)|=2g$.
This implies $|\fix(a,C_{2n})|=|\fix(\varphi_{\tilde{a,g}},D_n)|$.

Our remaining task is to check that (C) is satisfied. Before it, we have to compute $P=P(Q)$ and $P'=P(Q')$. 
\begin{itemize}
\item We see that $P=P(Q)=\langle 2g \rangle < C_{2n}$.  
\item By Proposition~\ref{prop:p}, we see that $P'= \langle \sigma^g \rangle$. 
\end{itemize}
Note that $P \cong_\mathrm{gr} P' \cong_\mathrm{gr} C_{n/g}$. 

We define $h : \langle 2g \rangle \rightarrow \langle \sigma^g \rangle$ as follows: 
$$h(2gj):=\sigma^{gj} \text{ for }j=0,1,\ldots,n/g-1. $$
It is trivial that $h$ is a group isomorphism. On (C-1), we see that 
\begin{align*}
h \circ a (2gj)=h(2agj)=\sigma^{agj} \;\text{ and }\; 
\varphi_{\tilde{a},g} \circ h(2gj)=\varphi_{\tilde{a},g}(\sigma^{gj})=\sigma^{\tilde{a}gj}=\sigma^{agj}. 
\end{align*}
On (C-2), for any $i \in C_{2n}$, we see that 
$$h(s_i(e))=h((1-a)i)=h\left(2g \frac{1-a}{2g} i\right)=\sigma^{\frac{1-a}{2}i}.$$
On the other hand, for $x'=\tau^\epsilon \sigma^j \in D_n$, we see from \eqref{eq:sxe} that 
$$s_{x'}'(e)=\sigma^{(-1)^\epsilon (1-\tilde{a})j+\epsilon g}=\sigma^{(-1)^\epsilon (1-a)j+\epsilon g}.$$
\begin{itemize}
\item When $g'=g$, the linear congruence $(1-a)j\equiv \frac{(1-a)i}{2} \pmod{n}$ has a solution $j=z_0$
by Lemma~\ref{hodai_D_n} (i).
By letting $\epsilon=0$ and $j=z_0$, we see that $s_{x'}'(e)=\sigma^{\frac{(1-a)i}{2}}=h(s_i(e))$.
\item When $g'=2g$, we see that $n$ must be even.
Then we can define the function $\mu$ like \eqref{eq:mu-def}.
Put $\ell =\frac{ki}{g}\in C_n$. Then $\ell + \mu(\ell)$ must be even.
In this case, the linear congruence $(-1)^{\mu(\ell)}(1-a)j\equiv -(\ell +\mu(\ell))g \pmod{n}$ has a solution $j=z_1$
by Lemma~\ref{hodai_D_n} (i).
By letting $\epsilon=\mu(\ell)$ and $j=z_1$, we see that $$s_{x'}'(e)=\sigma^{(-1)^{\mu(\ell)}(1-a)z_1+\mu(\ell)g}
=\sigma^{-(\ell +\mu(\ell))g+\mu(\ell)g}=\sigma^{-\ell g}=
\sigma^{-ki}=h(s_i(e)).$$
\end{itemize}
\end{proof}

\bigskip


\section{Generalized Alexander quandles arising from finite groups of small orders}\label{sec:group}

This section is devoted to solving Problem~\ref{toi} in the case where the order of $G$ is small 
by using Theorem~\ref{monosugoi} together with Corollaries~\ref{cor1}, \ref{cor2} and \ref{cor3}. 
More concretely, we determine $\Qc_{\GA}(n)$ for small integers $n$. 
For this purpose, we investigate the generalized Alexander quandles arising from the groups with its orders $p$ or $2p$ or $p^2$, where $p$ is a prime. 

\subsection{The case $|G|=p$}\label{sec:p}
Let $G$ be a group with $|G|=p$, where $p$ is a prime. Then it is very well known that $G \cong_\mathrm{gr} C_p$. 
Moreover, since $C_p$ is simple, it follows from Theorem~\ref{simple} that 
$\Qc_{\GA}(p)$ bijectively corresponds to $\gauto(C_p)$, which is isomorphic to $C_p^\times$ as groups. In particular, $|\Qc_{\GA}(p)|=p-1$.

\subsection{The case $|G|=2p$}\label{sec:2p}
Let $G$ be a group with $|G|=2p$, where $p$ is an odd prime. (Note that the case $p=2$ will be discussed in Subsection~\ref{sec:p^2}.) 
Then it is well known that $G \cong_\mathrm{gr} C_{2p}$ or $D_p$ (see, e.g., \cite[Theorem 4.19]{Rotman}). 
Thanks to Corollary~\ref{cor3}, it suffices to discuss only $Q(D_p, \psi)$, where $\psi \in \gauto(D_p)$. 
Hence, $\Qc_{\GA}(2p)=\Qc(D_p)$. 
On the other hand, we know by Proposition~\ref{prop:conj_class_D_n} that the conjugacy classes for $\gauto(D_p)$ bijectively correspond to 
$\{\varphi_{a,d} : a \in C_p^\times, \text{$d$ is a divisor of $g$}\}$, where $g=\gcd(p,1-a)$. Note that $C_p^\times=\{1,2,\ldots,p-1\}$. 
\begin{itemize}
\item When $a=1$, we have $g=p$, so the possible $d$'s are $1$ and $p=g$. 
\item When $a \neq 1$, we have $g=1$, so the possible $d$ is only $1=g$. 
\end{itemize}
Therefore, $\gauto(D_p)$ bijectively corresponds to $\{(a,0) : a \in C_p^\times\} \cup \{(1,1)\} \subset C_p^\times \times C_p$. 
Those give all different generalized Alexander quandles by (the third condition of) Corollary~\ref{cor2}. 
In particular, $|\Qc_{\GA}(2p)|=p$.

\subsection{The case $|G|=p^2$}\label{sec:p^2}
Let $G$ be a group with $|G|=p^2$, where $p$ is a prime. 
Then it is well known that $G \cong_\mathrm{gr} C_{p^2}$ or $G \cong_\mathrm{gr} C_p \times C_p$ (see, e.g., \cite[Corollary 4.5 and Exercise 2.68]{Rotman}). 

First, let us consider $\Qc(C_{p^2})$ and take $Q \in \Qc(C_{p^2})$. 
Since $P=P(Q)$ is a subgroup of $C_{p^2}$, we see that $P$ is isomorphic to $\{0\}$ or $C_p$ or $C_{p^2}$ as groups. 
Here, we note that $s_e$ is nothing but a group automorphism of $C_{p^2}$, which corresponds to an element of $C_{p^2}^\times$. 
Let $a \in C_{p^2}^\times$ correspond to $s_e$. Here, we see from the description $s_x(e)=(1-a)x$ and Lemma~\ref{hodai_D_n} (iii) that 
$P=\{gy : y \in C_{p^2}\} \cong_\mathrm{gr} C_{p^2/g}$, where $g=\gcd(1-a,p^2)$. 
\begin{itemize}
\item Note that $P=\{0\}$ if and only if $Q$ is a trivial quandle (see Remark~\ref{chuui_P}). 
\item Let $P \cong_\mathrm{gr} C_p$. Then $g=p$. Thus, $a=kp+1$ for some $1 \leq k \leq p-1$. 
For $a=kp+1$ and $a' = k'p+1$ with $1 \leq k,k' \leq p-1$, let $g=\gcd(1-a,n=p^2)$ and $g'=\gcd(1-a',n)$, respectively. 
Then we see that $g=g'=p$ and $a \equiv a' \pmod{n/g(=p)}$. 
Thus, $Q(C_{p^2},a) \cong_\mathrm{qu} Q(C_{p^2},a')$ for any $a =kp+1$ and $a' =k'p+1$ by Corollary~\ref{cor:Nelson}. 
\item Let $P \cong_\mathrm{gr} C_{p^2}$. Then $g=1$. Thus, each $a \in C_{p^2}$ with $g=1$ defines different quandles. 
\end{itemize}
Therefore, in this case, $\Qc(C_{p^2})$ bijectively corresponds to $$\{1\} \cup \{p+1\} \cup \{kp+r : 0 \leq k \leq p-1, 2 \leq r \leq p-1\} \subset C_{p^2}^\times.$$

Next, let us consider $\Qc(C_p \times C_p)$. 
Let $G=C_p \times C_p$ and regard $G$ as an additive group of a vector space $C_p^2$ over a prime field $C_p$. 
Then it is known that $\gauto(G) \cong_\mathrm{gr} \GL(2,p)$ and 
a set of representatives of conjugacy classes for $\gauto(G)$ is determined by rational canonical forms as follows: 
\begin{align}\label{eq:conj_pxp}
\{\lambda I : \lambda \in C_p^\times \} \cup 
\left\{ \begin{pmatrix} 0 & a\\ 1 & b \end{pmatrix} :  a\in C_p^\times, \;  b \in C_p\right\}, 
\end{align}
where $I=\begin{pmatrix} 1 &0 \\ 0 &1 \end{pmatrix}$
(cf. \cite[Chapter~6]{Herstein}, \cite{Feit,Macdonald}). Note that $I \in \gauto(G)$ corresponds to a trivial quandle. 
\begin{itemize}
\item For $\lambda \in C_p^\times$ with $\lambda \neq 1$, since $1 - \lambda \in C_p^\times$, we see that $\Im (1 - \lambda) I = C_p \times C_p$. 
Hence, by Corollary~\ref{cor1}, we obtain that each $\lambda \in C_p^\times$ defines different quandles. 
\item For $M=\begin{pmatrix} 0 & a\\ 1 & b \end{pmatrix}$ with $a\in C_p^\times$ and $b \in C_p$, let us consider $N:=I-M$. Then 
\begin{align*}
N=\begin{pmatrix} 1 &-a \\ -1 &1-b \end{pmatrix} \rightarrow \begin{pmatrix} 1 &-a \\ 0 &1-a-b \end{pmatrix}, \text{ i.e., }
\rank N=\begin{cases} 1 &\text{if }b=1-a, \\
2 &\text{if }b \neq 1-a. 
\end{cases} 
\end{align*}
If $b \neq 1-a$, then $N$ is regular, so $\Im \rho=\Im N = C_p \times C_p$. Hence, each $M$ defines different quandles. 
If $b= 1-a$, then $\Im \rho$ is the image of $\begin{pmatrix}1 \\ -1 \end{pmatrix}$, i.e., $P \cong_\mathrm{gr} C_p$. 
Moreover, since $M \begin{pmatrix}1 \\ -1 \end{pmatrix}=-a \begin{pmatrix}1 \\ -1 \end{pmatrix}$, we see that $Q(P,M) \cong_\mathrm{qu} Q(C_p,-a)$. 
Hence, each $a \in C_p^\times$ defines different quandles. 
\end{itemize}
Therefore, in this case, $\Qc(C_p \times C_p)$ bijectively corresponds to \eqref{eq:conj_pxp} itself.

Furthermore, from the above discussions, we see the following: 
\begin{itemize}
\item $Q(C_{p^2},1) \cong_\mathrm{qu} Q(C_p \times C_p, I)$ holds; 
\item $Q(C_{p^2},p+1) \cong_\mathrm{qu} Q(C_p \times C_p,M)$ holds, where $M=\begin{pmatrix} 0 & -1\\ 1 & 2 \end{pmatrix}$; 
\item $Q(C_{p^2},a) \cong_\mathrm{qu} Q(C_p \times C_p, M)$ does not hold for any $a \in C_{p^2}^\times$ and $M$ in \eqref{eq:conj_pxp} for other cases. 
In fact, we have $P(Q(C_{p^2},a))=C_{p^2}$, while $C_{p^2}$ cannot be a normal subgroup of $C_p \times C_p$ (cf. Proposition~\ref{normal_subgroup} (i)). 
\end{itemize}
In particular, $$|\Qc_{\GA}(p^2)|=|\Qc(C_{p^2})|+|\Qc(C_p \times C_p)|-(1+1)=2p^2-2p-1.$$

\subsection{$n=8$}\label{sec:8}
Let $G$ be a group with $|G|=8$. Then $G$ is isomorphic to 
$C_8$ or $C_4 \times C_2$ or $C_2 \times C_2 \times C_2$ or $D_4$ or the quaternion group $Q_8$ (see, e.g., \cite[p. 250]{Roman}). 
By Corollary~\ref{cor3}, we do not need to discuss $\Qc(C_8)$. 

In what follows, we discuss $\Qc(G)$ case-by-case. 

\medskip

\noindent
\underline{The case $G = C_4 \times C_2$}: It is known that $\gauto(C_4 \times C_2) \cong_\mathrm{gr} D_4$ whose is automorphism is given as follows: 
for $\sigma \in D_4$, we let $\psi_\sigma : C_4 \times C_2 \rightarrow C_4 \times C_2$ by setting $\psi_\sigma((i,j))=(i+2j,i+j)$, 
and for $\tau \in D_4$, we let $\psi_{\tau} : C_4 \times C_2 \rightarrow C_4 \times C_2$ by setting $\psi_\tau((i,j))=(-i,i+j)$. 
Then the set of representatives of the conjugacy classes for $\gauto(C_4 \times C_2) \cong_\mathrm{gr} D_4$ is as follows: 
$$\id, \;\; \psi_\sigma, \;\; \psi_\sigma^2, \;\; \psi_\tau \;\text{ and }\; \psi_\tau \circ \psi_\sigma.$$
For each $\psi \in \{\id, \psi_\sigma,\psi_\sigma^2, \psi_\tau,\psi_\tau\circ\psi_\sigma\}$, we compute the invariants of $Q(C_4 \times C_2,\psi)$ as in Table~\ref{tab:C_4xC_2}: 
\begin{table}[tbh]
    \centering
    \begin{tabular}{cccccc}
        &$\psi$ &$\ord \psi$  &$|\fix(\psi,G)|$ & $P$   &   $\psi |_P$  \\
        \toprule
       $Q^8_1$&$\id$    &1        &8 &$\{e\}$ &$\id$ \\
       $Q^8_2$&$\psi_\sigma$ &4 &2                & $C_2 \times C_2$ & $\begin{pmatrix} 0 &1 \\ 1 &0\end{pmatrix}$ \\
       $Q^8_3$&$\psi_\sigma^2$ &2 & 4                 & $C_2$ & $\id$ \\
       $Q^8_4$&$\psi_\tau$   &2 & 4                 & $C_2$ & $\id$ \\
       $Q^8_5$&$\psi_\sigma \circ \psi_\tau$ &2   & 4                 & $C_2$ & $\id$ \\
        \bottomrule
    \end{tabular}

\medskip

    \caption{Invariants for $Q(C_4\times C_2,\psi)$}\label{tab:C_4xC_2}
\end{table}

By Corollary~\ref{cor1}, we can check that 
$Q^8_3\cong_\mathrm{qu} Q^8_4\cong_\mathrm{qu} Q^8_5$.

\noindent
\underline{The case $G= C_2 \times C_2 \times C_2$}: It is known that $\gauto(C_2 \times C_2 \times C_2) \cong_\mathrm{gr} \GL(3,2)$ 
and a set of representatives of its conjugacy classes (i.e., rational canonical forms) is as follows: 
$$\begin{pmatrix}
    1 & 0 & 0 \\
    0 & 1 & 0 \\
    0 & 0 & 1
    \end{pmatrix}, \;
    \begin{pmatrix}
    1 & 0 & 0 \\
    0 & 0 & 1 \\
    0 & 1 & 0
    \end{pmatrix}, \; 
    \begin{pmatrix}
    0 & 0 & 1 \\
    1 & 0 & 0 \\
    0 & 1 & 0
    \end{pmatrix}, \; 
    \begin{pmatrix}
    0 & 0 & 1 \\
    1 & 0 & 1 \\
    0 & 1 & 1
    \end{pmatrix}, \; 
    \begin{pmatrix}
    0 & 0 & 1 \\
    1 & 0 & 0 \\
    0 & 1 & 1
    \end{pmatrix}, \; 
    \begin{pmatrix}
    0 & 0 & 1 \\
    1 & 0 & 1 \\
    0 & 1 & 0
    \end{pmatrix}. 
$$
Let $M_1,\ldots,M_6$ be the above matrices from left to right, respectively. 
Then we can compute the invariants of $Q(C_2 \times C_2 \times C_2, M_i)$ for $i=1,\ldots,6$ as in  Table~\ref{tab:C2xC2xC2}: 
\begin{table}[tbh]
    \centering
    \begin{tabular}{cccccc}
        &$\psi$ &$\ord \psi$ & $|\fix(\psi,G)|$ & $P$ & $\psi |_P$ \\
        \toprule
        $Q^8_6$&$M_1=\id$    &1        &8 &$\{e\}$ &$\id$ \\

        $Q^8_7$&$M_2$ &2 & 4 & $C_2$ & $\mathrm{id}$ \\
        $Q^8_8$&$M_3$ &3 & 2 & $C_2\times C_2 $ & $\begin{pmatrix}
            0 & 1\\
            1 & 1
        \end{pmatrix}$ \\
        $Q^8_9$&$M_4$ &4 & 2 & $C_2\times C_2$ & $\begin{pmatrix}
            0 & 1\\
            1 & 0
        \end{pmatrix}$  \\
        $Q^8_{10}$&$M_5$ &7 & 1 & $C_2\times C_2\times C_2$ & $M_5$ \\
        $Q^8_{11}$&$M_6$ &7 & 1 & $C_2\times C_2\times C_2$ & $M_6$ \\
        \bottomrule
    \end{tabular}
    
    \medskip
    
    \caption{Invariants for $Q(C_2\times C_2\times C_2,M_i)$}
    \label{tab:C2xC2xC2}
\end{table}

\noindent
By Corollary~\ref{cor1}, we conclude that those quandles are all different.

\noindent
\underline{The case $G=D_4$}: 
A set of representatives of conjugacy classes is $\varphi_{1,0},\varphi_{1,1},\varphi_{1,2},\varphi_{3,1}$ and $\varphi_{3,2}$ as in Example~\ref{ex:cong}. 
We see from Corollary~\ref{cor2} that $Q(D_4,\varphi_{1,2}) \cong_\mathrm{qu} Q(D_4,\varphi_{3,2})$ and other quandles are all different. 
We can compute the invariants for $Q(D_4,\varphi_{a,b})$ as in Table~\ref{tab:D4}: 
\begin{table}[tbh]
    \centering
    \begin{tabular}{cccccc}
        &$\psi$ &$\ord \psi$ & $|\fix(\psi,G)|$ & $P$ & $\psi |_P$ \\
        \toprule
        $Q^8_{12}$&$\varphi_{1,0}=\id$    &1        &8 &$\{e\}$ &$\id$ \\

        $Q^8_{13}$&$\varphi_{3,1}$ &2 & 2 & $C_4$ & $3 \in C_4^\times$ \\
        $Q^8_{14}$&$\varphi_{1,2},\varphi_{3,2}$ &2 & 4 & $C_2$ & $\mathrm{id}$ \\
        $Q^8_{15}$&$\varphi_{1,1}$ &4 & 4 & $C_4$ & $\mathrm{id}$ \\
        \bottomrule
    \end{tabular}

\medskip

    \caption{Invariants for $Q(D_4,\varphi_{a,b})$}\label{tab:D4}
\end{table}

\noindent
\underline{The case $G=Q_8$}: Let $Q_8=\{\pm 1, \pm i \pm j , \pm k\}\subset \HH$. 
Then it is known that $\gauto(Q_8) \cong_\mathrm{gr} \Sf_4$ and its set of representatives of conjugacy classes is as follows: 
\begin{align*}
    \psi_1\colon & \pm 1 \mapsto \pm 1, \;  \pm i \mapsto \pm i,\  \pm j \mapsto \pm j,\  \pm k \mapsto \pm k \quad (\psi_1=\mathrm{id}),\\
    \psi_2\colon & \pm 1 \mapsto \pm 1, \;  \pm i \mapsto \pm i,\  \pm j \mapsto \mp j,\  \pm k \mapsto \mp k \quad (\psi_2=(\cdot)^i),\\
    \psi_3\colon & \pm 1 \mapsto \pm 1, \;  \pm i \mapsto \pm j,\  \pm j \mapsto \pm i,\  \pm k \mapsto \mp k,\\
    \psi_4\colon & \pm 1 \mapsto \pm 1, \;  \pm i \mapsto \pm j,\  \pm j \mapsto \pm k,\  \pm k \mapsto \pm i,\\
    \psi_5\colon & \pm 1 \mapsto \pm 1, \;  \pm i \mapsto \pm j,\  \pm j \mapsto \mp i,\  \pm k \mapsto \pm k \quad (\psi_5=\psi_3\circ \psi_2). 
\end{align*}
We can compute the invariants for $Q(Q_8,\psi_i)$ for $i=2,\ldots,5$ as in Table~\ref{tab:Q8}. 
We also check whether \ref{G-normal} and \ref{iikanji} are satisfied or not. 
\begin{table}[tbh]
    \centering
    \begin{tabular}{cccccccc}
        &$\psi$ &$\ord \psi$ & $|\fix(\psi,G)|$ & $P$ & $\psi |_P$ & \ref{G-normal} & \ref{iikanji} \\
        \toprule
       $Q^8_{16}$&$\psi_1$ &1& 8 & 1 & $\mathrm{id}$ & T & T\\
        $Q^8_{17}$&$\psi_2$ &2 & 4 & $C_2$ & $\mathrm{id}$ & T & T \\
        $Q^8_{18}$&$\psi_3$ &2 & 2 & $C_4$ & $3 \in C_4^\times$ &T &T \\
        $Q^8_{19}$&$\psi_4$ &3 & 2 & $Q_8$ & $\psi_4$ &T &F \\
        $Q^8_{20}$&$\psi_5$ &4 & 4 & $C_4$ & $\mathrm{id}$ &T &T\\
        \bottomrule
    \end{tabular}

\medskip

    \caption{Invariants for $Q(Q_8,\psi_i)$}
    \label{tab:Q8}
\end{table}

Furthermore, by Theorem~\ref{monosugoi}, we can verify the following isomorphisms: 
\begin{align*}
&Q^{8}_1 \cong_{\mathrm{qu}} Q^{8}_6 \cong_{\mathrm{qu}} Q^{8}_{12} \cong_{\mathrm{qu}} Q^{8}_{16}; \quad
Q^{8}_{13} \cong_{\mathrm{qu}} Q^{8}_{18};\\
&Q^{8}_3 \cong_{\mathrm{qu}} Q^{8}_4 \cong_{\mathrm{qu}} Q^{8}_{5}\cong_{\mathrm{qu}} Q^{8}_{7}\cong_{\mathrm{qu}} Q^{8}_{14}\cong_{\mathrm{qu}} Q^{8}_{17};
\quad
Q^{8}_2 \cong_{\mathrm{qu}} Q^{8}_9;
\quad
Q^{8}_{15} \cong_{\mathrm{qu}} Q^{8}_{20}.
\end{align*}
By summarizing the above discussions, we obtain the following list (see Table~\ref{tab:n8}): 
\begin{table}[tbh]
    \centering
    \begin{tabular}{ccccc}
         &$\ord \psi$ & $|\fix(\psi,G)|$ & $P$ & $\psi |_P$  \\
        \toprule
        $Q^8_1$  &1 &8 & 1 & $\mathrm{id}$ \\
        $Q^8_{13}$  &2 & 2 & $C_4$ & $\times 3$ \\
        $Q^8_3$  &2 & 4 & $C_2$ & $\mathrm{id}$\\
        $Q^8_8$  &3 & 2 & $C_2\times C_2 $ & $\begin{pmatrix}
            0 & 1\\
            1 & 1
        \end{pmatrix}$ \\
        $Q^8_{19}$  &3 & 2 &$Q_8$ & $\psi_4$ \\
        $Q^8_2$  &4 &2 &$C_2\times C_2$ & $\begin{pmatrix}
            0 & 1\\
            1 & 0
        \end{pmatrix}$  \\
        $Q^8_{15}$ &4 &4 & $C_4$ & $\mathrm{id}$ \\
        $Q^8_{10}$  &7& 1 & $C_2\times C_2\times C_2$ & $M_5$\\
        $Q^8_{11}$  &7& 1 & $C_2\times C_2\times C_2$ & $M_6$ \\
        \bottomrule
    \end{tabular}

\medskip

    \caption{List of $\Qc_{\GA}(8)$}
    \label{tab:n8}
\end{table}

\subsection{$n=12$}\label{sec:12}
It is known that $G$ is isomorphic to $C_{12}$, $C_6 \times C_2$, $D_6$, $\Dic_3$ or $\Af_4$ (see, e.g., \cite[P. 252]{Roman}), 
where $\Dic_3$ is a dicyclic group whose presentation is as follows: $$\Dic_{3}=\langle a,b : a^6=1,\  b^2=a^3,\  b^{-1} a b =a^{-1} \rangle.$$
By Corollary~\ref{cor3}, we do not need to discuss $\Qc(C_{12})$. 

Our classification can be performed in the similar way to the case $n=8$.

\noindent
\underline{The case $G=C_6 \times C_2$}: It is known that $\gauto(C_6 \times C_2) \cong_\mathrm{gr} D_6$ whose automorphism is given as follows: 
for $\sigma \in D_6$, we let $\alpha_\sigma : C_6 \times C_2 \rightarrow C_6 \times C_2$ by setting $\alpha_\sigma((i,j))=(2i+3j,i+j)$, 
and for $\tau \in D_6$, we let $\alpha_\tau : C_6 \times C_2 \rightarrow C_6 \times C_2$ by setting $\alpha_\tau((i,j))=(-i,i+j)$. 
Then the set of the conjugacy classes for $\gauto(C_6 \times C_2) \cong_\mathrm{gr} D_6$ is as follows: 
$$\id, \;\; \alpha_\sigma, \;\; \alpha_\sigma^2, \;\; \alpha_\sigma^3, \;\; \alpha_\tau \;\text{ and }\; \alpha_\tau \circ \alpha_\sigma.$$ 
For each $\psi \in \{\id,\alpha_\sigma,\alpha_\sigma^2,\alpha_\sigma^3,\alpha_\tau,\alpha_\tau \circ \alpha_\sigma\}$, 
we compute the invariants of $Q(C_6 \times C_2, \psi)$ as in Table~\ref{tab:C6xC2}: 
\begin{table}[tbh]
    \centering
    \begin{tabular}{cccccc}
         &$\psi$ &$\ord \psi$ & $|\fix(\psi,G)|$ & $P$ & $\psi |_P$  \\
        \toprule
        $Q^{12}_1$ & $\id$ & 1 & 12 & 1 & $\mathrm{id}$  \\
        $Q^{12}_2$ & $\alpha_\tau$ & 2 & 2 & $C_6$ & $\times 5$\\
        $Q^{12}_3$ & $\alpha_\sigma^3$ & 2 & 4 & $C_3$ & $\times 2$ \\
        $Q^{12}_4$ & $\alpha_\tau \circ \alpha_\sigma$ & 2 & 6 & $C_2$ & $\mathrm{id}$ \\
        $Q^{12}_5$ & $\alpha_\sigma^2$ & 3 & 3 & $C_2\times C_2$ & $\begin{pmatrix} 0 &1 \\ 1 &1 \end{pmatrix}$\\
        $Q^{12}_6$ & $\alpha_\sigma$ & 6 & 1 & $C_6\times C_2$ & $\alpha_\sigma$ \\
        \bottomrule
    \end{tabular}

\medskip

    \caption{List of $\mathcal{Q}(C_6\times C_2)$}\label{tab:C6xC2}
\end{table}

\noindent
\underline{The case $G=D_6$}: Along Section~\ref{sec:D_n}, we compute the invariants of $\Qc(D_6)$ as in Table~\ref{tab:D6}:

\begin{table}[tbh]
    \centering
    \begin{tabular}{cccccl}
        &$\psi$ & $\ord \psi$ & $|\fix(\psi,G)|$ & $P$ & $\psi |_P$ \\
        \toprule
        $Q^{12}_7$ &$\varphi_{1,0}$ & 1 & 12 & 1 & $\mathrm{id}$ \\
        $Q^{12}_8$ &$\varphi_{5,1}$ & 2 & 2 & $C_6$ & $\times 5$ \\
        $Q^{12}_9$ &$\varphi_{5,2}$ & 2 & 4 & $C_3$ & $\times 2$ \\
        $Q^{12}_{10}$ &$\varphi_{1,3}$ & 2 & 6 & $C_2$ & $\mathrm{id}$ \\
        $Q^{12}_{11}$ &$\varphi_{1,2}$ & 3 & 6 & $C_3$ & $\mathrm{id}$ \\
       $Q^{12}_{12}$ &$\varphi_{1,1}$ & 6 & 6 & $C_6$ & $\mathrm{id}$ \\
        \bottomrule
    \end{tabular}
    
    \medskip
    
    \caption{List of $\mathcal{Q}(D_6)$}\label{tab:D6}
\end{table}

\noindent
\underline{The case $G=\Dic_3$}: 
Note that each element of $\Dic_3$ can be of the form $a^ib^\varepsilon$ for $0 \leq i \leq 5$ and $\varepsilon \in \{0,1\}$. 
It is known that $\gauto(\Dic_3) \cong_\mathrm{gr} D_6$ whose isomorphism is given as follows: for $\sigma \in D_6$, we let 
$\beta_\sigma : \Dic_3 \rightarrow \Dic_3$ by setting $\beta_\sigma(a^i b^\varepsilon)=a^{i+\varepsilon}b^\varepsilon$, 
and for $\tau \in D_6$, we let $\beta_\tau : \Dic_3 \rightarrow \Dic_3$ by setting $\beta_\tau(a^i b^\varepsilon)=a^{-i}b^\varepsilon$. 
Then the set of representatives of the conjugacy classes for $\gauto(\Dic_3) \cong_\mathrm{gr} D_6$ is as follows: 
$$\id, \;\; \beta_\sigma, \;\; \beta_\sigma^2, \;\; \beta_\sigma^3, \;\; \beta_\tau \;\text{ and }\; \beta_\tau \circ \beta_\sigma.$$
For each $\psi \in \{\id,\beta_\sigma,\beta_\sigma^2,\beta_\sigma^3,\beta_\tau,\beta_\tau\circ\beta_\sigma\}$, 
we compute the invariants of $Q(\Dic_3,\psi)$ as in Table~\ref{tab:Dic3}: 
\begin{table}[tbh]
    \centering
    \begin{tabular}{cccccccc}
        &$\psi$ & $\ord \psi$ & $|\fix(\psi,G)|$ & $P$ & $\psi |_P$ & \ref{G-normal} & \ref{iikanji} \\
        \toprule
       $Q^{12}_{13}$ & $\id$ &1 & 12 & 1 & $\mathrm{id}$ &T &T \\
        $Q^{12}_{14}$ & $\beta_\tau \circ \beta_\sigma$ &2 & 2 & $C_6$ & $\times 5$ &T &T \\
        $Q^{12}_{15}$ & $\beta_\tau$ &2 & 4 & $C_3$ & $\times 2$ &T &T \\
        $Q^{12}_{16}$ & $\beta_\sigma^3$ &2 & 6 & $C_2$ & $\id$ &T &T \\
        $Q^{12}_{17}$ & $\beta_\sigma^2$ &3 & 6 & $C_3$ & $\id$ &T &T \\
        $Q^{12}_{18}$ & $\beta_\sigma$ &6 & 6 & $C_6$ & $\id$ &T &T \\
        \bottomrule
    \end{tabular}

\medskip

\caption{List of $\mathcal{Q}(\Dic_3)$}\label{tab:Dic3}
\end{table}

\noindent
\underline{The case $G=\Af_4$}: It is known that $\gauto(\Af_4) \cong \Sf_4$. We compute the invariants of $\Qc(\Af_4)$ as follows:

\begin{table}[tbh]
    \centering
    \begin{tabular}{cccccccl}
        &$\psi$ &$\ord \psi$ & $|\fix(\psi,G)|$ & $P$ & $\psi |_P$ & \ref{G-normal} & \ref{iikanji}\\
        \toprule
       $Q^{12}_{19}$ & $\mathrm{id}$ & 1 & 12 & 1 & $\mathrm{id}$   & T & T\\
        $Q^{12}_{20}$ & $(\cdot)^{(12)}$ & 2 & 2 & $\Af_4$ & $(\cdot)^{(12)}$ & T &F \\
        $Q^{12}_{21}$ & $(\cdot)^{(12)(34)}$ & 2 & 4 & $C_2\times C_2$ & $\id$ & T &T \\
        $Q^{12}_{22}$ & $(\cdot)^{(123)}$ & 3 & 3 & $C_2\times C_2$ & $\begin{pmatrix} 0 &1 \\ 1 &1 \end{pmatrix}$ & T &T \\
        $Q^{12}_{23}$ & $(\cdot)^{(1234)}$ & 4 & 2 & $\Af_4$ & $(\cdot)^{(1234)}$ & T &F \\
        \bottomrule
    \end{tabular}
    
    \medskip
    
        \caption{List of $\mathcal{Q}(\Af_4)$}\label{tab:A4}
\end{table}

Furthermore, by Theorem~\ref{monosugoi} and using {\tt GAP} \cite{GAP}, we can verify the following isomorphisms: 
\begin{align*}
&Q^{12}_1 \cong_{\mathrm{qu}} Q^{12}_7 \cong_{\mathrm{qu}} Q^{12}_{13} \cong_{\mathrm{qu}} Q^{12}_{19}; \quad
Q^{12}_2 \cong_{\mathrm{qu}} Q^{12}_8 \cong_{\mathrm{qu}} Q^{12}_{14}; \\
&Q^{12}_3 \cong_{\mathrm{qu}} Q^{12}_9 \cong_{\mathrm{qu}} Q^{12}_{15}; \quad
Q^{12}_4 \cong_{\mathrm{qu}} Q^{12}_{10} \cong_{\mathrm{qu}} Q^{12}_{16}; \quad Q^{12}_5 \cong_{\mathrm{qu}} Q^{12}_{22}. 
\end{align*}
By summarizing the above discussions, we obtain the following list (see Table~\ref{tab:n12}): 
\begin{table}[tbh]
    \centering
    \begin{tabular}{ccccl}
        &$\ord \psi$ & $|\fix(\psi,G)|$ & $P$ & $\psi |_P$  \\
        \toprule
        $Q_1^{12}$ &1 & 12 & 1 & $\mathrm{id}$ \\
        $Q_2^{12}$ &2 & 2 & $C_6$ & $\times 5$ \\
        $Q^{12}_{20}$ &2 & 2 & $\Af_4$ & $(\cdot)^{(12)}$ \\
        $Q_3^{12}$ &2 & 4 & $C_3$ & $\times 2$ \\
        $Q^{12}_{21}$  &2 & 4 & $C_2\times C_2$ &$\mathrm{id}$ \\
        $Q_4^{12}$ &2 & 6 & $C_2$ & $\mathrm{id}$ \\
        $Q^{12}_5$ &3 & 3 & $C_2\times C_2$ & $\begin{pmatrix} 0 &1 \\ 1 &1 \end{pmatrix}$ \\
        $Q^{12}_{11}$ &3 & 6 & $C_3$ & $\mathrm{id}$ \\
        $Q^{12}_{23}$ &4 & 2 & $\Af_4$ & $(\cdot)^{(1234)}$ \\
        $Q^{12}_6$ &6 & 1 & $C_6\times C_2$ &$\alpha_\sigma$ \\
        $Q^{12}_{12}$ &6 & 6 & $C_6$ & $\mathrm{id}$ \\
        \bottomrule
    \end{tabular}

\medskip
    
        \caption{List of $\Qc_{\GA}(12)$}\label{tab:n12}
\end{table}

\subsection{$n=15$}\label{sec:15}
Let $G$ be a group with $|G|=15$. Then $G$ is isomorphic to $C_{15}$ (see, e.g., \cite[P. 252]{Roman}), i.e., $\Qc_{\GA}(15)=\Qc(C_{15})$. 
By Corollary~\ref{cor:Nelson}, it is straightforward to check that 
$\gauto(C_{15}) \cong_\mathrm{gr} C_{15}^{\times} =\{1,2,4,7,8,11,13,14\}$ bijectively corresponds to $\Qc(C_{15})$. 
In particular, $|\Qc_{\GA}(15)|=8$. 

\subsection{$n=16$}\label{sec:16}
In the case $n=16$, we encounter the problem. Namely, there are two groups $G$ and $G'$ 
together with their automorphisms $\psi \in \gauto(G)$ and $\psi' \in \gauto(G')$ such that 
we cannot distinguish whether $Q(G,\psi) \cong_\mathrm{qu} Q(G',\psi')$ or not by using Theorem~\ref{monosugoi}. 
In fact, let $G=C_2 \times Q_8$ and let $G'=(C_4 \times C_2) \rtimes C_2$. 
We take $\phi \in \gauto(Q_8) \cong_\mathrm{gr} \Sf_4$ with $\ord \phi=3$ and 
let $\psi \in \gauto(G)$ be defined by $\psi((i,x))=(i,\phi(x))$, 
and we also take $\psi' \in \gauto(G')$ with $\ord \psi'=3$. 
Actually, we see that both $Q(G,\psi)$ and $Q(G',\psi')$ do not satisfy \ref{iikanji}, so we cannot apply Theorem~\ref{monosugoi}, 
while all of the invariants we have developed coincide. 

Hence, we cannot verify whether those are isomorphic as quandles or not by the current methods. 

\begin{rem}
    From Tables~\ref{tab:C_4xC_2}--\ref{tab:n12},
    it seems that the isomorphism among finite generalized Alexander quandles with the same cardinality can be determined only by the pair of $P$ and $\psi|_P$.
    However, there exist quandles that are not isomorphic even if the pair of $P$ and $\psi|_P$ are equal.
    When $n=16$, $Q(D_8,1,2)$ and $Q(D_8,5,2)$ have $P=P'\cong_\mathrm{gr} C_4$ (i.e., $d=d'=2$)
    and $\varphi_{1,2}|_P=\varphi_{5,2}|_{P'}=\mathrm{id}$ (i.e., $1\equiv 5 \pmod{\frac{8}{2}}$).
    On the other hand,
    since $|\fix(\varphi_{1,2},D_4)|=8$ and $|\fix(\varphi_{5,2},D_4)|=4$,
    the two quandles $Q(D_8,1,2)$ and $Q(D_8,5,2)$ are not isomorphic.
\end{rem}

\end{document}